\documentclass[11pt]{amsart}

\usepackage{epigamath}

\usepackage[english]{babel}

\numberwithin{equation}{section}

\usepackage{amscd}

\newtheorem{thm}{Theorem}[section]
\newtheorem{prop}[thm]{Proposition}
\newtheorem{cor}[thm]{Corollary}

\newtheorem{lemma}[thm]{Lemma}

\newtheorem{claim}[thm]{Claim}

\theoremstyle{definition}

\newtheorem{defn}[thm]{Definition}

\newtheorem{notn}[thm]{Notation}

\newtheorem{rmk}[thm]{Remark}

\newtheorem{Empty}[thm]{}

\newcommand{\bi}{\begin{itemize}}  
\newcommand{\ei}{\end{itemize}}
\newcommand{\bp}{\begin{proof}}
\newcommand{\ep}{\end{proof}}

\def\AA{\mathbb{A}}
\def\CC{\mathbb{C}}

\def\GG{\mathbb{G}}

\def\PP{\mathbb{P}}

\def\QQ{\mathbb{Q}}
\def\RR{\mathbb{R}}

\def\ZZ{\mathbb{Z}}

\def\VV{\mathbb{V}}

\def\ov{\overline}

\def\bG{\mathbb{G}}
\def\bP{\mathbb{P}}
\def\al{\alpha}
\def\be{\beta}
\def\de{\delta}
\def\eps{\epsilon}

\def\la{\lambda}

\def\om{\omega}
\def\si{\sigma}

\def\De{\Delta}
\def\Ga{\Gamma}

\def\Si{\Sigma}

\def\cA{\mathcal{A}}

\def\cC{\mathcal{C}}

\def\cE{\mathcal{E}}
\def\cF{\mathcal{F}}

\def\cL{\mathcal{L}}

\def\cM{\mathcal{M}}

\def\cO{\mathcal{O}}

\def\cT{\mathcal{T}}
\def\cTT{\tilde{\mathcal{T}}}

\def\cS{\mathcal{S}}
\def\cU{\mathcal{U}}

\def\ocM{\overline{\mathcal{M}}}

\def\codim{\text{codim}}

\def\Bl{\text{\rm Bl}}

\def\Ext{\text{\rm Ext}}
\def\Sym{\text{\rm Sym}}

\def\Hom{\text{\rm Hom}}

\def\HH{\text{H}}

\def\LM{\overline{\text{LM}}}
\def\M{\overline{\text{\rm M}}}

\def\PGL{\text{\rm PGL}}

\def\Pic{\text{\rm Pic}}
\def\Proj{\text{\rm Proj}}

\def\Spec{\text{Spec}}

\def\cV{\mathcal{V}}

\def\wt{\text{weight}}
\def\ba{\mathbf{a}}

\def\dra{\dashrightarrow}
\def\hra{\hookrightarrow}

\def\ra{\rightarrow}

\newcommand{\sslash}{\mathbin{/\mkern-6mu/}}

\EpigaVolumeYear{4}{2020} \EpigaArticleNr{20} \ReceivedOn{May 7,
2020}
\InFinalFormOn{November 9, 2020}
\AcceptedOn{November 27, 2020}

\title{Exceptional collections on certain Hassett spaces}
\titlemark{Exceptional collections on certain Hassett spaces}

\author{Ana-Maria Castravet}
\address{Universit\'e Paris-Saclay, UVSQ, CNRS, Laboratoire de Math\'ematiques de Versailles, 78000, Versailles, France}
\email{ana-maria.castravet@uvsq.fr}

\author{Jenia Tevelev}
\address{Department of Mathematics and Statistics, University of Massachusetts Amherst, 710 North Pleasant Street, Amherst, MA 01003, USA, and 
Laboratory of Algebraic Geometry and its Applications, HSE, Moscow, Russia}
\email{tevelev@math.umass.edu}

\authormark{A.-M. Castravet and J. Tevelev}

\AbstractInEnglish{\footnotesize{We construct an $S_2\times S_n$ invariant full exceptional collection on Hassett spaces of weighted stable rational curves with $n+2$ markings and weights 
$(\frac{1}{2}+\eta, \frac{1}{2}+\eta,\frac{1}{n},\ldots,\frac{1}{n})$, for $0<\eta\ll1$. These Hassett spaces can be identified with the symmetric GIT quotients of
$(\PP^1)^n$ by the diagonal action of $\GG_m$, when $n$ is odd, and their Kirwan desingularization, when $n$ is even. The existence of such an exceptional  collection is one of the needed ingredients in order to prove the existence of a full $S_n$-invariant exceptional collection on $\ocM_{0,n}$. To prove exceptionality we use the method of windows in derived categories. To prove fullness we use previous work on the existence of invariant full exceptional collections on Losev--Manin spaces.}}



\TitleInFrench{Collections exceptionnelles sur certains espaces de Hassett}

\AbstractInFrench{\footnotesize{Nous construisons une collection exceptionnelle pleine $S_2\times S_n$--invariante sur les espaces de Hassett de courbes rationnelles stables pond\'er\'ees avec $n+2$ marquages et poids 
$(\frac{1}{2}+\eta, \frac{1}{2}+\eta,\frac{1}{n},\ldots,\frac{1}{n})$, pour $0<\eta\ll1$. Ces espaces de Hassett peuvent \^etre identifi\'es avec les quotients sym\'etriques (au sens de la TGI) de 
$(\PP^1)^n$ par l'action diagonale de $\GG_m$ lorsque $n$ est impair, et leur d\'esingularisation de Kirwan lorsque $n$ est pair. 
L'existence d'une telle collection exceptionnelle est l'un des ingr\'edients n\'ecessaires afin de prouver l'existence d'une collection exceptionnelle $S_n$-invariante pleine sur $\ocM_{0,n}$. Pour prouver le caract\`ere exceptionnel, nous utilisons la m\'ethode des fen\^etres dans les cat\'egories d\'eriv\'ees. Pour prouver que la collection est pleine, nous utilisons un travail pr\'ec\'edent portant sur l'existence de collections exceptionnelles pleines sur les espaces de Losev--Manin.}}

\acknowledgement{AMC was supported by NSF grants DMS-1529735 and DMS-1701752. JT was supported by NSF grants DMS-1303415, DMS-1701704, and 
by the HSE University Basic Research Program, Russian Academic Excellence Project 5-100. 
Parts of this paper were written while AMC was visiting the Institut des Hautes \'Etudes Scientifiques in France
and JT was visiting the Fields Institute in Toronto, Canada.}

\begin{document}


\removeabove{0.5cm}
\removebetween{0.5cm}
\removebelow{0.6cm}

\maketitle

\begin{prelims}

\DisplayAbstractInEnglish

\smallskip

\DisplayKeyWords

\smallskip

\DisplayMSCclass

\medskip

\languagesection{Fran\c{c}ais}

\smallskip

\DisplayTitleInFrench

\smallskip

\DisplayAbstractInFrench

\end{prelims}


\newpage

\setcounter{tocdepth}{2}

\tableofcontents


\section{Introduction}

A conjecture of Manin and Orlov states that Grothendieck-Knudsen moduli space $\ocM_{0,n}$ of stable, rational curves with $n$ markings admits a full, exceptional collection which is invariant (as a set) under the action of the symmetric group $S_n$ permuting the markings. The conjecture has been proved by the authors in \cite{CT_partII} by reducing it to the similar statement for several Hassett spaces, one of which is the space under consideration in this paper. While the proof presented in  \cite{CT_partII} for other needed Hassett spaces is valid in this particular case as well, it was not discussed in  \cite{CT_partII} and we prefer to give a different and much simpler proof here.

\smallskip

For a vector of rational weights $\mathbf{a}=(a_1,\ldots,a_n)$ with $0<a_i\le1$ and $\sum a_i>2$, the Hassett space
$\M_{\mathbf{a}}$ is the moduli space of weighted pointed stable rational curves, i.e.,~pairs $(C,\sum a_ip_i)$ with slc singularities, such that 
$C$ is a genus~$0$, at worst nodal, curve and the  $\QQ$-line bundle $\omega_C(\sum a_ip_i)$ is ample. For example, $\ocM_{0,n}=\M_{(1,\ldots,1)}$.
There exist birational reduction morphisms $\M_{\mathbf{a}}\to\M_{\mathbf{a'}}$ every time the weight vectors are such that $a_i\ge a_i'$ for every~$i$. 

\smallskip 

Understanding  the derived categories of the Hassett spaces $\M_{\mathbf{a}}$ was considered in the work of Ballard, Favero and Katzarkov \cite{BFK}, and earlier, for 
$\ocM_{0,n}$ in the work of Manin and Smirnov \cite{ManinSmirnov1} (see also \cite{Smirnov Thesis,ManinSmirnov2}). However, here we consider a modified question. If $\Ga_{\mathbf{a}}\subseteq S_n$ denotes the stabilizer of the set of weights ${\mathbf{a}}$, we ask whether there exists a full, $\Ga_{\mathbf{a}}$-invariant exceptional collection on $\M_{\mathbf{a}}$. Theorem \cite[Theorem 1.5] {CT_partII} reduces the existence of such collections on $\ocM_{0,n}$, as well as many other Hassett spaces 
$\M_{\mathbf{a}}$, to the following cases: 
\begin{itemize}
\item[(I)] The Losev--Manin spaces $\M_{\mathbf{a}}$, where ${\mathbf{a}}=(1,1,\epsilon,\ldots,\epsilon)$, $0<\epsilon\ll1$. 
\item[(II)] The Hassett spaces $\M_{p,q}$, for $p+q=n$ ($q\geq0$, $p\geq2$) having $p$ \emph{heavy} weights and $q$ \emph{light} weights with the following properties:
\begin{equation}\label{Mpq}
a_1=\ldots=a_p=a+\eta,\quad a_{p+1}=\ldots=a_n=\epsilon,\quad 
pa+q\epsilon=2,
\end{equation}
where  $0<\eta,\epsilon\ll1.$
\end{itemize}

\smallskip 

To reduce to the above cases, the authors were inspired by results of Bergstrom and Minabe \cite{BergstromMinabe, BergstromMinabeLM} that used reduction maps between Hassett spaces. The existence of a full, invariant, exceptional collection in case (I) was proved in \cite{CT_partI}. 
The work in \cite{CT_partII} proves the statement for the spaces $\M_{p,q}$ in (II) with $p\ge3$ and is the most difficult part of the argument. The current paper treats the spaces $\M_{p,q}$ in (II) with $p=2$. We emphasize that this case is not explicitly proved in \cite{CT_partII}. 
However, the proof for $p>2$ is valid even when $p=2$.  The proof for $p>2$ requires a lot of different comparisons between different Hassett spaces. 
Here we prove that this can be avoided when $p=2$. 
More precisely, the main space under consideration when $p=2$ is the following: 

\begin{notn}\label{Z}
Let $Z_N$ denote the Hassett moduli space of rational curves with markings $N\cup\{0,\infty\}$ with weights of markings $0$ and $\infty$ equal to 
$\frac{1}{2}+\eta$ and the markings from $N$ equal to $\frac{1}{n}$, where $0<\eta\ll1$. We also write $Z_n:=Z_N$ for $n=|N|$ when there is no ambiguity. 
\end{notn}
The condition on the weights is equivalent to the condition (\ref{Mpq}) for $p=2$ (in which case, $a=1-\frac{(n-2)\epsilon}{2}$). Explicitly, 
all light points may coincide with one another and one heavy point may coincide with at most $\lfloor\frac{n-1}{2}\rfloor$ heavy points. 
We have the following description:

\begin{thm}\label{Z description}
When $n$ is odd, the space $Z_n$ is isomorphic to the symmetric GIT quotient 
$Z_n=(\PP^1)^n\sslash_{\cO(1,\ldots,1)}\GG_m$, 
with respect to the diagonal action of $\GG_m$ on 
$(\PP^1)^n$, coming from  $\GG_m$ acting on $\PP^1$ by $z\cdot [x,y]=[zx,z^{-1}y]$.
When $n$ is even, $Z_n$ is isomorphic to the Kirwan desingularization of the same GIT quotient. 
\end{thm}

Theorem \ref{Z description} for $n$ odd is stated in \cite{Ha} within a more general set-up. 
Theorem \ref{Z description} for $n$ even is a direct consequence of \cite{Ha}. For the reader's convenience, 
we give the proofs in Lemma \ref{identify} and Lemma~\ref{identify2}. 

\smallskip

The group $S_2\times S_n$ acts on $Z_n$ by permuting $0$, $\infty$, and the markings from $N$ respectively. In a similar fashion, 
the Losev--Manin space $\LM_N$ (or $\LM_n$, for $n=|N|$) of dimension $(n-1)$ is the Hassett space with weights 
$(1,1,\epsilon,\ldots,\epsilon)$, 
with markings from $N\cup\{0,\infty\}$ with the weights of $0$, $\infty$ equal to $1$, while markings from $N$ are equal to $\epsilon$, with $0<\epsilon\ll1$. 
The space $\LM_N$ is isomorphic to an iterated blow-up of $\PP^{n-1}$ along points $q_1,\ldots, q_n$ in linearly general position, and all linear subspaces spanned by $\{q_i\}$. In particular, $\LM_n$ is a toric variety. The action of $S_n$ permuting the markings from $N$ corresponds to a relabeling of the points 
$\{q_i\}$, while the action of $S_2$, permuting $0$, $\infty$, corresponds, at the level of $\PP^{n-1}$, to a Cremona transformation with center at the points
$\{q_i\}$. There is a birational $S_2\times S_N$-equivariant morphism, reducing  the weights of $0$ and $\infty$:
$p: \LM_N\ra Z_N$.  
In particular, $Z_N$ is also a toric variety. Our main theorem is the following:

\begin{thm}\label{main}
The Hassett space $Z_n$
has a full exceptional collection which is invariant under the action of $(S_2\times S_n)$. 
In particular, the K-group $K_0(Z_n)$ is a permutation 
$(S_2\times S_n)$-module.
\end{thm}

Theorem \ref{main} is the immediate consequence of Theorem \ref{odd} (case of $n$ odd) and Theorem \ref{even} (case of $n$ even). 
We now describe the collections. 
\begin{defn}\label{tautological}
If $(\pi:\cU\ra\M, \si_1,\ldots, \si_n)$ is the  universal family over the Hassett space $\M$, 
one defines tautological classes 
$$\psi_i:=\si_i^*\om_{\pi},\quad \de_{ij}=\si_i^*\si_j.$$
Note that when $n$ is odd, we have $\psi_0+\psi_{\infty}=0$ on $Z_n$. For other relations, including the case when $n$ is even, see
Section \ref{Hassett}.
\end{defn}

\begin{defn}\label{L odd case}
Assume $n$ is odd. 
Let $E\subseteq N$ and $p\in\ZZ$, such that if $e=|E|$ we have that $p+e$ is even. We define line bundles on $Z_n$ as follows:
$$L_{E,p}:=-\left(\frac{e-p}{2}\right)\psi_{\infty}-\sum_{j\in E}\de_{j\infty}.$$
As sums of $\QQ$-line bundles,  
$L_{E,p}=\frac{p}{2}\psi_{\infty}+\frac{1}{2}\sum_{j\in E} \psi_j=-\frac{p}{2}\psi_0+\frac{1}{2}\sum_{j\in E} \psi_j$. 
In particular, the action of $S_2$ exchanges $L_{E,p}$ with $L_{E,-p}$. 
The line bundles $L_{E,p}$ are natural from the GIT point of view, see (\ref{odd translate}).  
\end{defn}

\begin{thm}\label{odd}
Let $n=2s+1$ odd. The line bundles $\{L_{E,p}\}$ (Definition \ref{L odd case}) form a full, $(S_2\times S_n)$ invariant exceptional collection in $D^b(Z_n)$ under the condition:
$$|p|+\min(e,n-e)\leq s,\quad \text{where}\quad e=|E|,\quad p+e\quad \text{even}.$$
The line bundles are ordered by decreasing $e$, and for a fixed $e$, arbitrarily. 
\end{thm}

The collection in Theorem \ref{odd} is the dual of the collection in \cite[Theorem 1.10]{CT_partII} 
for $p=2$, with some of the constraints on the order removed.
See also Remark \ref{elaborate1} for a more precise statement.

\smallskip

Consider now the case when $n=2s+2\geq2$ is even. In this case the universal family over $Z_n$ has reducible fibers. For each partition 
$N=T\sqcup T^c$, $|T|=|T^c|=s+1$, we denote $\de_{T\cup\{\infty\}}\subseteq Z_n$ the boundary component parametrizing nodal rational curve with two components, with markings from $T\cup\{\infty\}$ on one component and $T^c\cup\{0\}$ on the other. Moreover, $\de_{T\cup\{\infty\}}=\PP^s\times\PP^s$ and
we have that $Z_n\ra (\PP^1)^n\sslash_{\cO(1,\ldots,1)}\PGL_2$ 
is a Kirwan resolution of singularities with exceptional divisors $\de_{T\cup\{\infty\}}$. 

\begin{defn}\label{L even case}
Assume $n$ is even. 
Let $E\subseteq N$ and $p\in\ZZ$, such that if $e=|E|$ we have that $p+e$ is even. We define line bundles on $Z_n$ as follows: 
$$L_{E,p}:=-\left(\frac{e-p}{2}\right)\psi_{\infty}-\sum_{j\in E}\de_{j\infty}-\sum_{|E\cap T|-\frac{e-p}{2}>0}\left(|E\cap T|-\frac{e-p}{2}\right)\de_{T\cup\{\infty\}}.$$
The line bundles $L_{E,p}$ are natural from the GIT point of view, see Definition \ref{even translate} and the discussion thereafter. From this point of view,  it is also clear that the action of $S_2$ exchanges $L_{E,p}$ with $L_{E,-p}$. 
\end{defn}

\begin{thm}\label{even}
Assume $n=2s+2$ is even, $s\geq0$. The following form a full, $(S_2\times S_n)$ invariant exceptional collection in $D^b(Z_n)$: 
\bi
\item The torsion sheaves $\cO(-a,-b)$ supported on $\de_{T\cup\{\infty\}}=\PP^s\times\PP^s$, for all $T\subseteq N$, $|T|=|T^c|=s+1$, such that one of the following holds:
\bi
\item $0<a\leq s$, $0<b\leq s$,
\item $a=0$, $0<b<\frac{s+1}{2}$,
\item $b=0$, $0<a<\frac{s+1}{2}$. 
\ei
\item The line bundles $\{L_{E,p}\}$ (Definition \ref{L even case})
under the following condition:
$$|p|+\min(e,n+1-e)\leq s+1,\quad \text{where}\quad e=|E|,\quad p+e\quad \text{even}.$$
\ei
The order is as follows: all torsion sheaves precede the line bundles, the torsion sheaves are arranged in order of decreasing $(a+b)$, while the line bundles are arranged in order of decreasing $e$, and for a fixed $e$, arbitrarily. 
\end{thm}

The torsion part of the collection in Theorem \ref{even} is the same as the torsion part of the collection in \cite[Theorem 1.15]{CT_partII} for $p=2$.
However, the remaining parts are not the same, nor are they dual to each other, as in the case of Theorem \ref{odd}. 
There is a relationship between the dual collection $\{L^\vee_{E,p}\}$ and the torsion free part of the collection  in \cite[Theorem 1.15] {CT_partII} for $p=2$, but 
this is more complicated -- see Remark \ref{elaborate2} for a precise statement. 

\smallskip

To prove that our collections are exceptional, we use the method of windows \cite{DHL,BFK}. We then use some of the main results of 
\cite[Proposition 1.8, Theorem 1.10]{CT_partI} to prove fullness, by using the reduction map $p:\LM_n\ra Z_n$ in order to compare our collections on $Z_n$ with 
with the push forward of the full exceptional collection on the Losev--Manin space. We emphasize that while in \cite{CT_partII} we prove exceptionality and fullness
on spaces like $Z_N$ indirectly, by working on their contractions (small resolutions of the singular GIT quotient when $n$ is even), 
in this paper we prove both exceptionality and fullness directly, by using the method of windows 
(for $n$ even on the Kirwan resolution, the blow-up of the strictly semistable locus).

\smallskip

As remarked in \cite{CT_partI}, we do not know any smooth projective toric varieties $X$ with an action of a finite group 
$\Gamma$ normalizing the torus action which do not have a $\Gamma$-equivariant exceptional collection $\{E_i\}$ of maximal 
possible length (equal to the topological Euler characteristic of $X$).
From this point of view, the  Losev--Manin spaces $\LM_N$ and their birational contractions $Z_N$ provide evidence that this may be true in general. The existence of such a collection implies that the K-group $K_0(X)$ is a permutation $\Gamma$-module.
In the Galois setting (when $X$ is defined over a field which is not algebraically closed
and $\Gamma$ is the absolute Galois group), an analogous statement was conjectured by Merkurjev and Panin \cite{MP}.
Of course one may further wonder if $\{E_i\}$ is in fact full, which is related to (non)-existence of phantom categories on $X$,
another difficult open question. 

\smallskip

We refer to \cite{CT_Duke,CT_Crelle,CT_rigid} for background information on the birational geometry of $\ocM_{0,n}$, the Losev--Manin space
and other related spaces. 

\smallskip

\noindent{\bf Organization of paper.}
In Section 2 we discuss preliminaries on Hassett spaces and prove some general results on how 
tautological classes pull back under reduction morphisms. These results are of independent interest and have been already used in a crucial way in 
\cite{CT_partII}. In Section 3, we discuss the GIT interpretation of the Hassett spaces $Z_n$ in the $n$ odd case and prove Theorem \ref{odd}. 
In Section 4, we do the same for the $n$ even case and prove Theorem \ref{even}. Section 5 serves as an appendix, recalling results on Losev--Manin spaces from \cite{CT_partI} and calculating the push forward to $Z_n$ of the full exceptional collection on the Losev--Manin space $\LM_n$. These results are used in Sections 3 and 4 to prove fullness in Theorems \ref{odd} and \ref{even}. Throughout the paper, we do not distinguish between line bundles and the corresponding divisor classes.

\smallskip

\noindent{\bf Acknowledgements.} 
We are grateful to Alexander Kuznetsov for suggesting the problem about the derived categories of moduli spaces of pointed curves in the equivariant setting. 
We thank Daniel Halpern--Leistner for his help with windows in derived categories. We thank Valery Alexeev and the anonymous referee for useful comments.


\section{Preliminaries on Hassett spaces}\label{Hassett}

We refer to \cite{Ha} for background on the Hassett moduli spaces. 
Recall that for a choice of weights 
$${\mathbf{a}}=(a_1, \ldots, a_n),\quad a_i\in\QQ,\quad 0<a_i\leq 1,\quad\sum a_i>2,$$ 
we denote by  $\M_{\mathbf{a}}$ the fine moduli space of weighted rational curves with $n$ markings which are stable with respect to the set of weights $\ba$.  
Moreover, $\M_\ba$ is a smooth projective variety of dimension $(n-3)$. 
Note that the polytope of weights has a chamber structure with walls 
$\sum_{i\in I} a_i=1$ for every subset $I\subseteq \{1,\ldots,n\}$. 
One obtains the Losev--Manin space $\LM_N$ by considering weights on the set of markings $\{0,\infty\}\cup N$: 
 $$\big(1,1,\frac{1}{n},\ldots,\frac{1}{n}\big),\quad n=|N|.$$ 
 Replacing the weights equal to $\frac{1}{n}$ with some $\epsilon\in\QQ$, for some $0<\epsilon\ll 1$,  defines the same 
 moduli problem, hence, gives isomorphic moduli spaces. 
 
Similarly, the moduli space $Z_N$ of Notation \ref{Z} is the moduli space with set of markings $\{0,\infty\}\cup N$ and weights 
$$\big(\frac{1}{2}+\eta,\frac{1}{2}+\eta,\frac{1}{n},\ldots,\frac{1}{n}\big),\quad \eta\in\QQ,\quad 0<\eta\ll 1.$$ 

If $\ba=(a_1, \ldots, a_n)$ and $\ba'=(b_1, \ldots, b_n)$ are such that $a_i\geq b_i$, for all $i$, there is a reduction 
morphism $\rho: \M_\ba\ra\M_\ba'$. This is a birational morphism whose exceptional locus consists of boundary divisors $\de_{I}$ (parametrizing reducible curves with a node that disconnects the markings from $I$ and $I^c$) for every subset $I\subseteq N$ such that $\sum_{i\in I} a_i>1$,  but 
$\sum_{i\in I} b_i\leq1$. For us a special role will be played by the reduction map
$p:\LM_N\ra Z_N$ which reduces the weights of $\{0,\infty\}$ from $1$ to the minimum possible. 

\smallskip

For a Hassett space $\M=\M_\ba$, with universal family $(\pi:\cU\ra\M, \{\si_i\})$, recall that 
we define $\psi_i:=\si_i^*\om_{\pi}$, $\de_{ij}=\si_i^*\si_j$. 
Since the sections $\sigma_i$ lie in the locus where 
the map $\pi$ is smooth, the identity $\sigma_i\cdot\omega_{\pi}=-\sigma_i^2$  holds on $\cU$. Therefore, 
$-\psi_i=\pi_*\big(\sigma_i^2\big)=\sigma_i^*\sigma_i$. 

\begin{lemma}\label{relations}
Assume $\M$ is a Hassett space whose universal family $\pi:\cU\ra\M$ is a $\PP^1$-bundle. Then the identity 
$-\om_{\pi}=2\si_i+\pi^*(\psi_i)$ holds on $\cU$, and therefore, on $\M$ we have for all $i\neq j$:
$$\psi_i+\psi_j=-2\de_{ij}.$$
Hence, for all distinct $i,j,k$, we have $\psi_i=-\de_{ij}-\de_{ik}+\de_{jk}$. 
\end{lemma}

\begin{proof}
Indeed, $-\om_{\pi}-2\si_i$ restricts to the fibers of the $\bP^1$-bundle trivially, and
therefore is of the form $\pi^*(L)$ for some line bundle on $\M$. Pulling back
by $\si_i$ shows that $L=\psi_i$.
\end{proof}

When $n$ is odd, the universal family $\cU\ra Z_N$ is a $\PP^1$-bundle and the sections $\si_0$ and $\si_{\infty}$ are distinct. 
Lemma \ref{relations} has the following:
\begin{cor}
The following identities hold on $Z_N$ when $n$ is odd:
\begin{equation}\label{P1-bundle identities}
\psi_0=-\psi_{\infty}=-\de_{i0}+\de_{i\infty},\quad \psi_i=-\de_{i0}-\de_{i\infty}.
\end{equation}
\end{cor}

\begin{lemma}\label{general statement}
Let $\M=\M_\ba$, $\M'=\M_{\ba'}$ be Hassett spaces, with $\ba=(a_i)$, $\ba'=(b_i)$, $a_i\geq b_i$ for all $i$. Consider the corresponding reduction map 
$p:\M'\ra\M$. 
Let $(\pi:\cU\ra\M, \{\sigma_i\})$, $(\pi':\cU'\ra\M',\{\sigma'_i\})$ be the universal families.  
Denote by
$(\rho: \cV\ra \M', \{s_i\})$ the pull-back of  $(\pi':\cU\ra\M,\{\sigma_i\})$ to $\M'$. 
Then there exists a commutative diagram: 
\begin{equation*}
\begin{CD}
\cU'      @>v>>  \cV@>q>> \cU\\
@VV{\pi'}V        @V{\rho}VV @V{\pi}VV \\
\M'   @>Id>>  \M'  @>{p}>>  \M
\end{CD}
\end{equation*}

Furthermore, identifying $\cU'$ with a Hassett space $\M_{\tilde{\ba}}$, where $\tilde{\ba}=(a_1,\ldots, a_n, 0)$ 
(with an additional marking $x$ with weight $0$) \cite[2.1.1]{Ha}, we have:
$$v^*\omega_{\rho}=\omega_{\pi'}-\sum_{|I|\geq2, \sum_{i\in I}a_i>1, \sum_{i\in I}b_i\leq1} \de_{I\cup\{x\}},$$
$$v^*s_i=\sigma_i+\sum_{i\in I, |I|\geq2,  \sum_{i\in I}a_i>1, \sum_{i\in I}b_i\leq1} \de_{I\cup\{x\}},$$
$$p^*\psi_i=\psi_i-\sum_{i\in I, |I|\geq2, \sum_{i\in I}a_i>1, \sum_{i\in I}b_i\leq1}\de_I,$$ 
$$p^*\de_{ij}=\de_{ij}+\sum_{i,j\in I, |I|\geq 3,\sum_{i\in I}a_i>1, \sum_{i\in I}b_i\leq1}\de_I.$$
\end{lemma}

\bp
The spaces $\cU$ and $\cU'$ are smooth \cite[Propositions 5.3 and 5.4]{Ha}.
The existence of the commutative diagram follows from semi-stable reduction \cite[Proof of Theorem 4.1]{Ha}. 
The map $v$ is obtained by applying the relative MMP for the line bundle $\omega_{\pi'}(\sum b_i\sigma'_i)$. Concretely, the relative MMP
results in a sequences of blow-downs, followed by a small crepant map: 
$$\cU'=\cS^1\ra\cS^2\ra\ldots\ra\cS^r=\cV,$$
(all over $\M'$). 
The resulting map $v:\cU'\ra\cV$ is a birational map which contracts divisors in $\cU'$ to 
codimension $2$ loci in $\cV$ 
(as the relative dimension drops from $1$ to $0$). Note that $\cV$  is generically smooth along these loci.
The $v$-exceptional divisors can be identified via $\cU'\cong\M_{\tilde{\ba}}$ with 
boundary divisors $\de_{I\cup\{x\}}$ ($I\subseteq N$), with the property that
$\sum_{i\in I} a_i>1$, $\sum_{i\in I} b_i\leq1$.  

For a flat family of nodal curves $u: \cC\ra B$ with Gorenstein base $B$ (in our case smooth)
the relative dualizing sheaf $\omega_u$ is a line bundle on $\cC$ with first Chern class 
$K_{\cC}-u^*K_B$, where $K_{\cC}$ and $K_B$ denote the corresponding canonical divisors. 
In particular:
$$\omega_{\pi'}=K_{\cU'}-{\pi'}^*K_{\M'}, \quad \omega_{\rho}=K_{\cV}-\rho^*K_{\M'}.$$
Since the map $v$ on an open set is the blow-up of codimension $2$ loci in $\cV$, it follows that 
$K_{\cU'}=v^*K_{\cV}+\sum E,$ by the blow-up formula. Hence,
$v^*\omega_{\rho}=\omega_{\pi'}-\sum E$, 
where the sum runs over all prime divisors $E$ which are $v$-exceptional.
This proves the first identity. For the second, we identify the sections $\sigma'_i$ (resp.,
$\sigma_i$) with the boundary divisors $\de_{ix}$ in $\cU'$ (resp., in $\cU$). Note that the proper transform of the section $s_i$ is $\sigma'_i$ and $s_i$ contains $v(\de_{I\cup\{x\}})$ ($|I|\geq2$), for $\de_{I\cup\{x\}}$ $v$-exceptional if and only if 
$i\in I$. Moreover, in this case, $v(\de_{I\cup\{x\}})$ is contained in $s_i$ (with codimension $1$) and $s_i$ is smooth (since $\M'$ is). The second identity follows. By Definition \ref{tautological} and the diagram, 
$$p^*\psi_i=p^*\sigma_i^*\omega_{\pi}=s_i^*q^*\omega_{\pi}=s_i^*\omega_{\rho}=
{\sigma'_i}^*v^*\omega_{\rho},$$
$$p^*\de_{ij}=p^*\sigma_i^*(\sigma_j)=s_i^*q^*(\sigma_j)=s_i^*s_j=
{\sigma'}_i^*v^*s_j.$$
The last two formulas now follow using the first two and the fact that ${\sigma'}_i^*\de_{I\cup\{x\}}=\de_I$ if $i\in I$ and is $0$ otherwise. 
\ep

\begin{cor}\label{pull by p}
Let $p:\LM_N\ra Z_N$ be the reduction map. Let $s:=\Bigl\lfloor\frac{n-1}{2}\Bigr\rfloor$. 
Then
$$p^*\psi_0=\psi_0-\sum_{I\subseteq N, 1\leq |I|\leq s}\de_{I\cup\{0\}},$$
$$p^*\psi_i=-\sum_{i\in I\subseteq N, 1\leq |I|\leq s}\big(\de_{I\cup\{0\}}+\de_{I\cup\{\infty\}}\big) \quad (i\in N),$$
$$p^*\de_{i0}=\sum_{i\in I\subseteq N, 1\leq |I|\leq s}\de_{I\cup\{0\}}\quad (i\in N),$$ 
$$p^*\de_{ij}=\de_{ij}+\sum_{i,j\in I\subseteq N, 1\leq |I|\leq s}\big(\de_{I\cup\{0\}}+\de_{I\cup\{\infty\}}\big)  \quad (i,j\in N).$$
 \end{cor}

\begin{lemma}\label{psi_i is zero on LM}
On the Losev--Manin space $\LM_N$, we have $\psi_i=0$ for all $i\in N$. 
\end{lemma}

\bp
Apply Lemma  \ref{general statement} to a reduction map $p: \M_{0,N\cup\{0,\infty\}}\ra\LM_N$:
$$p^*\psi_i=\psi_i-\sum_{i\in I, |I|\geq2, 0,\infty\in I^c}\de_I.$$
The right hand side of the equality is $0$ \cite[Lemma 3.4]{KeelTevelev}. Therefore, $p^*\psi_i=0$. As  $p_*\cO=\cO$, by the projection formula, we have 
$\psi_i=0$.
\ep 

\bp[Proof of Corollary \ref{pull by p}]
Follows from Lemma \ref{general statement} and Lemma \ref{psi_i is zero on LM}. In the notation of Lemma \ref{general statement}, the universal family 
$\cU'$ over $\M'=\LM_N$ can be identified with $\M_{\tilde{\ba}}$, where 
$\tilde{\ba}=(1,1,\epsilon,\ldots,\epsilon,0)$, with an additional marking $x$ with $0$ weight.
But $\cU'$ can also be identified with $\LM_{N\cup\{x\}}=\M_{(1,1,\epsilon,\ldots\epsilon)}$
(here $x$ has weight $\epsilon$). Via this identification, boundary divisors $\de_J$ correspond to boundary divisors $\de_J$, 
for any $J\subseteq N\cup\{0,\infty,x\}$. 
The $v$-exceptional divisors appearing in the sum are $\de_{I\cup\{x,0\}}$, $\de_{I\cup\{x,\infty\}}$, $I\subseteq N$, $|I| \leq \Bigl\lfloor\frac{n-1}{2}\Bigr\rfloor$. 
\ep

When $n=|N|$ is even, the Hassett space $Z_N=\M_{(\frac{1}{2}+\eta, \frac{1}{2}+\eta,\frac{1}{n},\ldots,\frac{1}{n})}$ of Notation \ref{Z}
is closely related to the following Hassett spaces:
$$Z'_N=\M_{(\frac{1}{2}+\epsilon, \frac{1}{2},\frac{1}{n},\ldots,\frac{1}{n})},\quad Z''_N=\M_{(\frac{1}{2}, \frac{1}{2}+\epsilon,\frac{1}{n},\ldots,\frac{1}{n})},$$
with  weights assigned to $(\infty, 0, p_1, \ldots, p_n)$. 
There exist 
$p': Z_N\ra Z'_N$, $p': Z_N\ra Z''_N$, reduction maps 
that contract the boundary divisors using the two different projections. 
The universal families over $Z'_N$ and $Z''_N$ are $\PP^1$-bundles. Lemma  \ref{general statement} applied to the reduction maps $p'$, $p''$ leads to: 

\begin{lemma}\label{relations2}
Assume $n=|N|$ is even. The following relations hold between the tautological classes on the Hassett space $Z_N$: 
$$\psi_0=\de_{i\infty}-\de_{i0}+\sum_{i\in T, |T|=\frac{n}{2}}\de_{T\cup\{\infty\}},\quad \psi_{\infty}=\de_{i0}-\de_{i\infty}+\sum_{i\notin T, |T|=\frac{n}{2}}\de_{T\cup\{\infty\}},$$ 
$$\psi_0+\psi_{\infty}=\sum_{|T|=\frac{n}{2}}\de_{T\cup\{\infty\}}.$$ 
\end{lemma}

\bp
The second relation follows from the first using the $S_2$ symmetry, while the third follows by adding the first two. 
To prove the first relation, consider the reduction map $p': Z_N\ra Z'_N$. To avoid confusion, we denote by $\psi'_i$, $\de'_{ij}$ 
(resp., $\psi_i$, $\de_{ij}$) the tautological classes on $Z'_N$ (resp., on $Z_N$). The universal family $\cC'\ra Z'_N$ is a 
$\PP^1$-bundle. By Lemma \ref{relations}, we have $\psi'_{\infty}=\de'_{i0}-\de'_{i\infty}$ (since $\de'_{0\infty}=0$). The relation follows, as by 
Lemma \ref{general statement}, we have
$${p'}^*\psi'_{\infty}=\psi_{\infty}-\sum_{|T|=\frac{n}{2}}\de_{T\cup\{\infty\}},\quad {p'}^*\de'_{i\infty}=\de_{i\infty}+\sum_{i\in T, |T|=\frac{n}{2}}\de_{T\cup\{\infty\}},\quad {p'}^*\de'_{i0}=\de_{i0}.$$ 
\ep


\section{Proof of Theorem \ref{odd}}\label{odd section}

We start with a few generalities on GIT quotients $(\bP^1)^n_{ss}\sslash\bG_m$. 
For $n$ odd, we first show that the Hassett space $Z_N$ introduced in (\ref{Z}) can be identified with symmetric
GIT quotients $(\bP^1)^n_{ss}\sslash\bG_m$. We use the method of windows from \cite{DHL} to prove exceptionality of the collections in Theorem \ref{odd}.  We then prove that the collection is full, by using the full exceptional collection on the Losev--Manin spaces $\LM_N$ (see Section \ref{LM}). 

\subsection{Generalities on GIT quotients $(\bP^1)^n_{ss}\sslash\bG_m$} Assume $n$ is an arbitrary positive integer. 
Let $\GG_m=\Spec\, k[z,z^{-1}]$ act on $\AA^2$ by 
$z\cdot(x,y)=(zx,z^{-1}y)$. 
Let $P\bG_m:=\bG_m/\{\pm1\}$. Note that $P\bG_m$ acts on $\bP^1$ faithfully.
Let $0\in\bP^1$ be the point with homogeneous coordinates $[0:1]$ and let $\infty=[1:0]$.

\smallskip

We use concepts of ``linearized vector bundles'' and ``equivariant vector bundles'' interchangeably.
For (complexes of) coherent sheaves, we prefer ``equivariant''.
We endow the line bundle $\cO_{\bP^1}(-1)$ with a $\bG_m$-linearization induced by the above action of $\bG_m$
on its total space $\VV\cO_{\bP^1}(-1)\subset\PP^1\times\AA^2$. 

\smallskip

Consider the diagonal action of $\bG_m$ on $(\bP^1)^n$.
For $\bar j=(j_1,\ldots,j_n)$ in $\ZZ^n$, we denote $\cO(\bar j)$ the  line bundle $\cO(j_1,\ldots,j_n)$ on $(\bP^1)^n$
with $\bG_m$-linearization given by the tensor product of linearizations above.
We denote $\cO\otimes z^k$ the trivial line bundle 
with $\bG_m$-linearization given by the character $\GG_m\ra\GG_m$, $z\mapsto z^k$. 
For every equivariant coherent sheaf $\cF$  (resp., ~a complex of sheaves $\cF^\bullet$),
we denote by $\cF\otimes z^k$ (resp., ~$\cF^\bullet\otimes z^k$) the tensor product with $\cO\otimes z^k$.
Note that $\cO(\bar j)\otimes z^k$ is $P\bG_m$-linearized iff $j_1+\ldots+j_n+k$ is even.

There is an action of $S_2\times S_n$ on $(\bP^1)^n$  which normalizes the $\bG_m$ action.
Namely, $S_n$ permutes the factors of $(\bP^1)^n$ and $S_2$ acts on $\bP^1$ by $z\mapsto z^{-1}$.
This action permutes  linearized line bundles  $\cO(\bar j)\otimes z^k$ as follows: $S_n$ permutes components of  $\bar j$
and $S_2$ flips $k\mapsto-k$.

\begin{notn}
Consider the GIT quotient
$$\Si_n:=(\bP^1)^n_{ss}\sslash_{\cL}\bG_m,\quad \cL=\cO(1,\ldots,1),$$
with respect to the ample line bundle $\cL$ (with its canonical $\bG_m$-linearization described above). 
Here $(\bP^1)^n_{ss}$ denotes the semi-stable locus with respect to this linearization. Let
$\phi: (\PP^1)_{ss}^n\ra \Si_n$ denote the canonical morphism. 
\end{notn}
As GIT quotients $X\sslash_{\cL}~G$ are by definition $\Proj \big(R(X, \cL)^G\big)$, where $R(X, \cL)^G$ is the invariant part of the section ring $R(X, \cL)$, 
we may replace $\cL$ with any positive multiple. As the action of $P\GG_m$ on $(\PP^1)^n$ is induced from  the action of $\GG_m$, 
$\Si_n$ is isomorphic to the GIT quotient $(\bP^1)^n_{ss}\sslash P\GG_m$ (with respect to any even multiple of $\cL$). 
The action of $S_2\times S_n$ on $(\bP^1)^n$ descends to $\Si_n$. 

By the Hilbert-Mumford criterion, a point $(z_i)$ in $(\PP^1)^n$ is semi-stable (resp., stable) if $\leq\frac{n}{2}$ (resp., $<\frac{n}{2}$) 
of the $z_i$ equal $0$ or equal $\infty$. 

\subsection{The space $Z_N$ as a GIT quotient when $n$ is odd} 
When $n$ is odd, there are no strictly semistable points and the action of $P\bG_m$ on $(\bP^1)^n_{ss}$ is free. 
In particular, $\Si_n$ is smooth and by Kempf's descent lemma, any $P\bG_m$-linearized line bundle on $(\bP^1)^n_{ss}$ 
descends to a line bundle on $\Si_n$. Furthermore, $\Si_n$ can be identified with the quotient stack $[(\bP^1)^n_{ss}/P\bG_m]$ 
and its derived category $D^b(\Si_n)$ with the equivariant derived category $D^b_{P\bG_m}((\bP^1)^n_{ss})$.

Consider the trivial $\PP^1$-bundle on $ (\PP^1)^n$ with the following sections:
$$\rho: (\PP^1)^n\times \PP^1=\Proj(\Sym(\cO\oplus\cO))\ra (\PP^1)^n,$$
$$s_0(\ov z)= (\ov z, 0), \quad s_{\infty}(\ov z)= (\ov z, \infty), \quad s_i(\ov z)= (\ov z, pr_i(\ov z)),$$ 
where $pr_i:(\PP^1)^n\ra\PP^1$ is the $i$-th projection. The sections $s_0$, resp., $s_{\infty}$ are induced by the projection 
$p_2: \cO\oplus\cO\ra\cO$, resp., $p_1: \cO\oplus\cO\ra\cO$, while the section $s_i$ is induced by the map
$\cO\oplus\cO\ra pr_i^*\cO(1)$ given by the sections 
$x_i=pr_i^*x, y_i=pr_i^*y$ of $pr^*\cO(1)$ that define $0$ and $\infty$ on the $i$-th copy of $\PP^1$. 
\begin{notn}\label{Delta}
Let $\De_{i0}=pr_i^{-1}(\{0\})\subseteq(\PP^1)^n$ and $\De_{i\infty}=pr_i^{-1}(\{\infty\})\subseteq(\PP^1)^n$. 
\end{notn}
Note that  $\De_{i0}$ is the zero locus of the section $x_i$, or the locus in $(\PP^1)^n$ where $s_i=s_0$. 
Similarly, let $\De_{i\infty}$ the zero locus of the section $y_i$. 

We now endow all the above vector bundles with $\bG_m$-linearizations. Let 
$$\cL_0=\cO\otimes z,\quad \cL_{\infty}=\cO\otimes z^{-1}, \quad \cL_i=pr_i^*\cO(1)\otimes 1,\quad \cE=\cL_0\oplus \cL_{\infty}.$$
The maps $\cL_0\ra\cL_i$, $\cL_{\infty}\ra\cL_i$ (given by the sections $x_i$, $y_i$) are $\bG_m$-equivariant, hence, induce $\bG_m$-equivariant surjective
maps $\cE\ra\cL_i$. The projection maps $\cE\ra\cL_0$ and $\cE\ra\cL_{\infty}$ are clearly $\bG_m$-equivariant. While none of $\cE$, $\cL_0, \cL_{\infty}, \cL_i$ are 
$P\GG_m$-linearized vector bundles, tensoring with $\cO(1,\ldots, 1)$ solves this problem, and we obtain a non-trivial $\PP^1$-bundle
$\pi: \PP(\cE)\ra \Si_n$ with disjoint sections $\si_0$, $\si_{\infty}$ and additional sections $\si_1,\ldots, \si_n$. 

Denote $\de_{i0}$ the locus in $\Si_n$ where $\si_i=\si_0$. This is the zero locus of the section giving the map $\cL_{\infty}\ra\cL_i$ on $\Si_n$, i.e., the section whose pull-back to $(\PP^1)^n$ is the section $x_i$. Similarly, we let  $\de_{i\infty}$ the locus in $\Si_n$ where $\si_i=\si_{\infty}$. Hence, the sections $x_i$, $y_i$ of $pr_i^*\cO(1)\otimes 1$ defining $\De_{i0}$, $\De_{i\infty}$ descend to global sections of the corresponding line bundle on $\Si_n$ and define 
$\de_{i0}$, $\de_{i\infty}$.  

\begin{lemma}\label{dictionary}
Assume $n$ is odd. We have the following dictionary between line bundles on the GIT quotient $\Si_n$ and $P\GG_m$-linearized
line bundles on $(\bP^1)^n$: 
$$\cO(\de_{i0})=pr^*_i\cO(1)\otimes z,\quad \cO(\de_{i\infty})=pr^*_i\cO(1)\otimes z^{-1}$$
$$\psi_0=\cO\otimes z^{-2},\quad \psi_{\infty}=\cO\otimes z^{2},\quad \psi_i=pr^*_i\cO(-2)\otimes 1.$$
\end{lemma}

\bp
The first two formulas follows from the previous discussion: $\cO(\de_{i0})$ corresponds  to the $P\GG_m$-linearized
line bundle $\cL_i\otimes \cL_{\infty}^\vee$. The remaining formulas follow from Lemma \ref{identify} and the identities (\ref{P1-bundle identities}).
\ep

\begin{lemma}\label{identify}
If $n=|N|$ is odd, the Hassett space $Z_N$ (see Notation \ref{Z}) 
is isomorphic to the GIT quotient $\Si_n=(\bP^1)^n_{ss}\sslash_{\cO(1,\ldots,1)}\bG_m.$ 
\end{lemma}

\bp
The trivial $\PP^1$-bundle $\rho: (\PP^1)^n_{ss}\times \PP^1\ra (\PP^1)^n_{ss}$ with sections $s_0$, $s_{\infty}$, $s_i$ is the pull-back of 
the $\PP^1$-bundle $\pi: \PP(\cE)\ra \Si_n$ and sections $\si_0$, $\si_{\infty}$, $\si_i$. Since the former is a family of $\cA$-stable rational curves, where 
$\cA=(\frac{1}{2}+\eta, \frac{1}{2}+\eta,\frac{1}{n},\ldots,\frac{1}{n})$, we have an induced morphism $f: \Si_n\ra Z_N$.  Clearly, every $\cA$-stable pointed rational curve is represented in the family over $(\PP^1)^n_{ss}$ (hence, $\Si_n$). Furthermore, two elements of this family are isomorphic if and only if they belong to the same orbit under the action of $\GG_m$. It follows that $f$ is one-to-one on closed points. As both $Z_N$ and $\Si_n$ are smooth, $f$ must be an isomorphism. Alternatively, there is an induced morphism 
$F: (\PP^1)^n_{ss}\ra Z_N$ which is $\GG_m$-equivariant (with $\GG_m$ acting trivially on $Z_N$). As $\Si_n$ is a categorical quotient, it follows that $F$ factors through $\Si_n$ and as before, the resulting map $f: \Si_n\ra Z_N$ must be an isomorphism.  
\ep

\subsection{Exceptionality}
When $n$ is odd, $\Si_n$ is a smooth polarized projective  toric variety for the torus $\bG_m^{n-1}$ and  its polytope is a cross-section
of the $n$-dimensional cube (the polytope of $(\bP^1)^n$ with respect to $\cL$) by the hyperplane normal to and bisecting the big diagonal.
In particular, the topological Euler characteristic $e(\Si_n)$ is equal to the number of 
edges of the hypercube intersecting that hyperplane:
$$e(\Si_n)=n{n-1\choose {n-1\over 2}}=n{n\choose 0}+(n-2){n\choose 1}+(n-4){n\choose 2}+\ldots.$$

By Lemma \ref{dictionary}, the line bundles $\{L_{E,p}\}$ in Theorem \ref{odd} correspond to restrictions to $(\bP^1)^n_{ss}$
of $P\GG_m$ linearized line bundles on $(\bP^1)^n$ 
\begin{equation}\label{odd translate}
L_{E,p}=\cO(-E)\otimes z^p,
\end{equation}
where $\cO(-E)=\cO(\bar j)$, and $\bar j$ is a vector of $0$'s and $(-1)$'s, with $-1$'s corresponding to the indices in $E\subseteq N$. (Here we abuse notations
and we denote by $L_{E,p}$ both the line bundle on $(\bP^1)^n$  and the corresponding one on $\Si_n$.) 
The collection is $(S_2\times S_n)$-equivariant and consists of  $e(\Si_n)$ line bundles. 

\bp[Proof of Theorem \ref{odd} -- exceptionality]
Let $G:=P\bG_m$.
We use the method of windows \cite{DHL}.
We describe the Kempf--Ness stratification \cite[Section 2.1] {DHL}   
of the unstable locus $(\bP^1)^n_{us}$ with respect to $\cL$. 
The $G$-fixed points are
$$Z_I=\{(x_i)\,|\,x_i=0\ \hbox{\rm for}\ i\not\in I,\ x_i=\infty\ \hbox{\rm for}\ i\in I\}$$
for every subset $I\subseteq\{1,\ldots,n\}$.
Let $\sigma_I:\,Z_I\hookrightarrow (\bP^1)^n$ be the inclusion map.
The stratification comes from an ordering of the pairs $(\lambda, Z)$, where $\lambda: \GG_m\ra G$ is a $1$-PS  and 
$Z$ is a connected component of the $\lambda$-fixed locus (the points $Z_I$ in our case). The ordering is such that the function
$$\mu(\lambda, Z)=-\frac{\hbox{\rm weight}_\lambda\cL|_{Z}}{|\lambda|},$$ 
is decreasing. Here $|\lambda|$ is an Euclidean norm on $\Hom(\GG_m,G)\otimes_{\ZZ}\RR$. We refer to \cite[Section 2.1] {DHL} 
for the details. As $\mu(\lambda, Z)=\mu(\lambda^k, Z)$ for any integer $k>0$, it follows that, in our situation, one only has to consider 
pairs $(\lambda, Z_I)$ and $(\lambda', Z_I)$, for the two $1$-PS $\lambda(z)=z$ and  $\lambda'(z)=z^{-1}$. Recall that 
$$\hbox{\rm weight}_\lambda\cO(-1)|_{\infty}=+1,\quad \hbox{\rm weight}_\lambda\cO(-1)|_{0}=-1,$$
$$\hbox{\rm weight}_\lambda(\cO\otimes z^p)|_q=p\quad \text{ for all points }q\in\PP^1.$$
It follows that $\hbox{\rm weight}_{\lambda'}\cO(-1)|_{\infty}=-1$, $\hbox{\rm weight}_{\lambda'}\cO(-1)|_{0}=+1$ and 
$$\hbox{\rm weight}_\lambda\cL|_{Z_I}=|I^c|-|I|, \quad \hbox{\rm weight}_{\lambda'}\cL|_{Z_I}=-|I^c|+|I|.$$
The unstable locus is the union of the following Kempf--Ness strata: 
$$S_I=\{(x_i)\,|\,x_i=\infty\ \hbox{\rm  if}\ i\in I, x_i\neq\infty\ \hbox{\rm  if}\ i\notin I\}\cong \AA^{|I^c|}\quad \text{for}\quad  |I|>n/2,$$
$$S'_I=\{(x_i)\,|\,x_i=0\ \hbox{\rm if}\ i\not\in I,  x_i\neq0\ \hbox{\rm  if}\ i\in I\}\simeq \AA^{|I|}\quad \text{for}\quad  |I|<n/2.$$
The destabilizing $1$-PS for $S_I$ (resp.~for $S'_I$) is $\lambda$ (resp.~$\lambda'$).
The  $1$-PS $\lambda$ (resp., $\lambda'$) acts on the conormal bundle $N^\vee_{S_I|(\PP^1)^n}$ (resp., $N^\vee_{S'_I|(\PP^1)^n}$) restricted to $Z_I$ with positive weights 
and their sum $\eta_I$ (resp., $\eta'_I$) can be computed as
$$\eta_I=2|I|,\quad\hbox{\rm resp.}\quad \eta'_I=2|I^c|.$$
To see this, note that the sum of $\lambda$-weights of $\big(N^\vee_{S_I|(\PP^1)^n}\big)_{|Z_I}$  equals 
$$\hbox{\rm weight}_\lambda \big(\det N^\vee_{S_I|(\PP^1)^n}\big)_{|Z_I}=
\hbox{\rm weight}_\lambda \big(\det T_{S_I}\big)_{|Z_I}-\hbox{\rm weight}_\lambda \big(\det T_{(\PP^1)^n}\big)_{|Z_I}.$$
Note that $S_I$ can be identified with $\AA^{|I^c|}$ and the point $Z_I\in S_I$ with the point $0\in\AA^{|I^c|}$. The action of $G$ on 
$\AA^{|I^c|}$ is via $z\cdot (x_j)=(z^2x_j)$. It follows that $\hbox{\rm weight}_\lambda {T_{S_I}}_{|Z_I}=2|I^c|$. Similarly, the tangent space
$\big(\det T_{(\PP^1)^n}\big)_{|Z_I}$ can be identified with the tangent space of $T_0\AA^n$, with the action of $G$ on
$(x_j)\in \AA^n$ being $z\cdot x_j=z^2x_j$ if $j\in I^c$ and $z\cdot x_j=z^{-2}x_j$ if $j\in I$. 
It follows that $\hbox{\rm weight}_\lambda \big(T_{(\PP^1)^n}\big)_{|Z_I}=2|I^c|-2|I|$. Hence, $\eta_I=2|I|$. Similarly, $\eta'_I=2|I^c|$. 

For the Kempf-Ness strata  $S_I$ and $S'_I$ we make a choice of ``weights''
$$w_I=w'_I=-2s,\quad \text{where}\quad n=2s+1.$$

By the main result of \cite[Theorem 2.10]{DHL}, $D^b_G((\bP^1)^n_{ss})$
is equivalent to the window $\GG_w$ in the equivariant derived category $D^b_G((\bP^1)^n)$, namely the full subcategory
of all complexes of equivariant sheaves $\cF^\bullet$ such that all weights (with respect to corresponding destabilizing $1$-PS)
of the cohomology sheaves of the complex $\sigma_I^*\cF^\bullet$ lie in the segment
$$[w_I,w_I+\eta_I)\quad\hbox{\rm or}\quad [w'_I,w'_I+\eta'_I),\quad\hbox{\rm respectively.}$$

We  prove that the window $\GG_w$ contains all linearized line bundles $L_{E,p}=\cO(-E)\otimes z^p$ 
from Theorem \ref{odd}. Recall that $n=2s+1$. Since the collection is $S_2$ invariant and $S_2$ flips the strata $S_I$ and $S'_I$,
it suffices to check the window conditions for $S_I$.
The $\lambda$-weight of $\cO(-E)\otimes z^p$ restricted to $Z_I$ equals $|I\cap E|-|I^c\cap E|+p$. 
It is straightforward to check that the maximum of this quantity over all $E$ is equal to $2s+2|I|-n+1$ when $s$ is odd, 
or $2s+2|I|-n-1$ when $s$ is even, and the minimum to $-2s$, hence the claim.
Since our  collection of linearized line bundles is clearly an exceptional collection on $D^b_G((\bP^1)^n)$,
it follows it is an exceptional collection in $D^b_G(Z_n)$. 
\ep

\subsection{Fullness}
We will prove the following general statement. 
\begin{thm}\label{full odd}
The collection in Theorem \ref{odd} generates all line bundles 
$$L_{E,p}:=\cO(-E)\otimes z^p,$$
for all $E\subseteq N$, $e=|E|$, $p\in\ZZ$ with $e+p$ even. 
\end{thm}

\bp[Proof of Theorem \ref{odd} - fullness]
By Theorem \ref{full odd}, the collection in Theorem \ref{odd} generates all the objects $Rp_*({\pi^*_I}\hat\GG)$ from Corollary \ref{rewrite}. Fullness then follows by Corollary \ref{S is enough}. Alternatively, it is easy to see that line bundles $L_{E,p}$
generate the derived category of the stack $[(\bP^1)^n/P\GG_m]$ and we can finish as
in \cite[Proposition 4.1]{CT_partII}.
\ep

\bp[Proof Theorem \ref{full odd}]
For simplicity, denote by $\cC$ the collection in the theorem. We introduce the \emph{score} of a pair $(E,p)$, with $e=|E|$ as 
$$s(E,p):=|p|+\min\{e,n-e\}.$$
The collection $\cC$ consists of $L_{E,p}$ with $s(E,p)\leq s$. We prove the statement by induction on the score $s(E,p)$, and for equal score, by induction on $|p|$. 

Let $(E,p)$ be any pair as in Theorem \ref{full odd}. If $s(E,p)\leq s$, there is nothing to prove. Assume 
$s(E,p)>s$. Using $S_2$-symmetry, we may assume  w.l.o.g. that $p\geq0$. 
We will use the two types of $P\GG_m$-equivariant Koszul resolutions from Lemma \ref{G Koszul} to successively generate all objects. 

\smallskip

\underline{Case $e\leq s$.} 
The sequence (1) in Lemma \ref{K} for a set $I$ with $|I|=s+1$ followed by tensoring with $L_{E,p}=\cO(-E)\otimes z^p$, gives
an exact sequence
$$
0\ra L_{E\cup I,p-s-1}\ra\ldots\ra\bigoplus_{J\subseteq I, |J|=j}L_{E\cup J,p-j}\ra\ldots\ra L_{E,p}\ra0.$$
We prove that each term $L_{E\cup J,p-j}$ is generated by $\cC$ for all $j>0$. Note that  $s(E,p)=|p|+e=p+e$.
If $p-j\geq0$, then 
$$s(E\cup J, p-j)\leq (p-j)+(e+j)=p+e=s(E,p),$$ 
but as $p-j<p$, we are done by induction on $|p|$. If $p-j<0$ then 
$$s(E\cup J, p-j)\leq (j-p)+n-(e+j)=n-e-p<e+p=s(E,p)$$
since we assume $e+p>s$. In particular, $L_{E\cup J,p-j}$ is in $\cC$. 

\smallskip 

\underline{Case $e\geq s+1$.} Let $I\subseteq E$, with $|I|=s+1$. The sequence (2) in Lemma \ref{K} for the set $I$, followed by tensoring with 
$L_{E',p-s-1}=\cO(E')\otimes z^{p-s-1}$, where $E'=E\setminus I$, 
gives an exact sequence 
$$0\ra L_{E,p}\ra\ldots\ra\bigoplus_{J\subseteq I, |J|=j}L_{E'\cup J,j+p-s-1}\ra\ldots\ra L_{E',p-s-1}\ra0.$$
We prove that each term $L_{E\cup J,j+p-s-1}$ is generated by $\cC$ for all $J\neq I$ (when $(E'\cup J,j+p-s-1)=(E,p)$).  Note that  
$s(E,p)=p+n-e$. We let $e':=|E'|=e-s-1$.  If $j+p-s-1\geq0$, then 
$$s(E'\cup J, j+p-s-1)\leq (j+p-s-1)+(n-e'-j)=p+n-e=s(E,p).$$
As $p+j-s-1\leq p$ with equality if and only if $J=I$, we are done by induction on $|p|$. If $j+p-s-1<0$, then 
$$s(E'\cup J, j+p-s-1)\leq -(j+p-s-1)+(e'+j)=e-p<s(E,p)=p+n-e,$$
since we assume $s(E,p)>s$, which gives $e-p\leq s$. 
\ep

\begin{lemma}\label{G Koszul}\label{K}
Let $n=2s+1$, $I\subseteq N$, $|I|=s+1$. There are two types of $P\GG_m$-equivariant resolutions:
\bi
\item[(1)] The restriction to $(\PP^1)^n_{ss}$ of the Koszul complex of the intersection of the divisors $\De_{i0}$ (Notation \ref{Delta}) for $i\in I$, which takes the form 
$$0\ra \cO(-I)\otimes z^{-(s+1)}\ra\ldots\ra\bigoplus_{J\subseteq I, |J|=j}\cO(-J)\otimes z^{-j}\ra\ldots\ra\cO\otimes 1\ra0$$
\item[(2)] The restriction to $(\PP^1)^n_{ss}$ of the Koszul complex of the intersection of the divisors $\De_{i\infty}$ (Notation \ref{Delta}) for $i\in I$, which takes the form 
$$0\ra \cO(-I)\otimes z^{(s+1)}\ra\ldots\ra\bigoplus_{J\subseteq I, |J|=j}\cO(-J)\otimes z^j\ra\ldots\ra\cO\otimes 1\ra0$$
\ei
\end{lemma}

\bp Let $G=P\GG_m$. 
Denote for simplicity $D_i=\De_{i0}$, for all $i\in N$. The divisors $D_1,\ldots, D_n$ intersect with simple normal crossings. 
Let $Y_I:=\cap_{i\in I}D_i\subseteq (\PP^1)^n$.  
Consider the Koszul resolution of $Y_I$:
$$\ldots \ra\oplus_{i<j, i,j\in I}\cO(-D_i-D_j)\ra  \oplus_{i\in I}\cO(-D_I)\ra\cO\ra\cO_{Y_I}\ra0.$$
Each of these maps in the sequence is a direct sum of maps of the form 
$$\cO(-D_{j_1}-\ldots-D_{j_t})\ra \cO(-D_{j_1}-\ldots-D_{j_{t-1}})$$
obtained by multiplication with a canonical section corresponding to the effective divisor $D_{j_t}$. 
This can be made into a $G$-equivariant map:
$$\cO(-D_{j_1}-\ldots-D_{j_t})\otimes z^{-t}\ra \cO(-D_{j_1}-\ldots-D_{j_{t-1}})\otimes z^{-(t-1)}.$$
since $\cO(-D_i)\otimes z^{-1}\ra\cO$ is the $G$-equivariant map given by multiplication with $x_i$, whose zero locus 
is $D_i=\De_{i0}$ (see Lemma \ref{dictionary} and the discussion preceding it). The Lemma follows by restriction to $(\PP^1)^n_{ss}$. Note that $Y_I\cap(\PP^1)^n_{ss}=\emptyset$. 
The proof of (2) is similar, with the only difference that multiplication with $y_i$, the canonical section of $\De_{i\infty}$ corresponds to a $G$-equivariant
map $\cO(-\De_{i\infty})\otimes z\ra\cO$. 
\ep

\begin{rmk}\label{elaborate1}
We explain the connection with case $p=2$, $q=n=2s+1$ of \cite[Theorem 1.10]{CT_partII}. 
The collection there is the following:

(i) The line bundles $F_{0,E}:=-\frac{1}{2}\sum_{j\in E}\psi_j$ ($e=|E|$ is even)
 in the so-called group $1$ (group $1A$ and group $1B$ of the theorem coincide in this case).

(ii) The line bundles in the so-called group $2$: 
$$\cT_{l,\{u\}\cup E}:=\si_u^*\big(\om_{\pi}^{\frac{e+1-l}{2}}(E\cup\{u\})\big)=
\frac{e-l-1}{2}\psi_u+\sum_{j\in E}\de_{ju}=-\frac{l+1}{2}\psi_u-\sum_{j\in E}\frac{1}{2}\psi_j$$ 
where $e=|E|$, $u\in\{0,\infty\}$, $l\geq0$, $l+|E\cap\{u\}|$ even (i.e., $l+e$ odd), with
$$l+\min\{e,n-e\}\leq s-1.$$

This collection is the dual of the one in Theorem \ref{odd}.
The elements in group $2$  with $l=p-1$, $u=\infty$  
recover the dual of the collection in Theorem \ref{odd} when $p>0$. Similarly, 
elements in group $2$ with $l=-p-1$, $u=0$ recover the dual of the collection in Theorem \ref{odd} when $p<0$. The elements of 
group $1$ recover the dual of the collection in Theorem \ref{odd}  when $p=0$. 
\end{rmk}


\section{Proof of Theorem \ref{even}}

We employ a similar strategy as in Section \ref{odd section}. We identify the Hassett space 
$Z_N$ (see (\ref{Z})) when $n=|N|$ is even with the Kirwan resolution of the symmetric GIT quotient $\Si_n$. 
We use the method of windows \cite{DHL} to prove the exceptionality part of Theorem \ref{even}. We prove fullness using previous results on Losev--Manin 
spaces $\LM_N$ (see Section \ref{LM}). 

\subsection{The space $Z_N$ as a GIT quotient, $n$ even}\label{even intro} 
Assume $n=2s+2$. There are $n\choose s+1$ strictly semistable points $\{p_T\}\in(\PP^1)_{ss}^n$ one for each subset $T\subseteq N$, $|T|=s+1$.
More precisely, the point $p_T$ is obtained by taking $\infty$ for spots in $T$ and $0$ for spots in $T^c$.  
Instead of the GIT quotient $\Si_n$, which is singular at the images of these points, we consider its Kirwan resolution $\tilde\Si_n$
constructed as follows. 

Let $W=W_n$ be the blow-up of $(\bP^1)^n$ at the points $\{p_T\}$ and let $\{E_T\}$ be the corresponding exceptional divisors. 
The action of $\GG_m$ lifts to $W$. To describe this action locally around a point $p_T$, assume for simplicity
$T=\{s+2,\ldots, n\}$ around the point $p_T$. Consider the affine chart 
$$\AA^n=(\PP^1\setminus\{\infty\})^{s+1}\times(\PP^1\setminus\{0\})^{s+1}$$ 
In the new coordinates, we have $p_T=0=(0,\ldots,0)$. We let 
$((x_i), (y_i))$, resp., $((t_i), (u_i))$, for $i=1,\ldots, s$, 
be coordinates on $\AA^n$, resp., $\PP^{n-1}$.  
Then $W$ is locally
the blow-up $\Bl_0\AA^n$, with equations 
$$x_it_j=x_jt_i,\quad  x_iu_j=y_jt_i,\quad y_iu_j=y_ju_i.$$ 
The action of $\GG_m$ on $W$ is given by
$$z\cdot\big((x_i,y_i),[t_i,u_i]\big)=\big((z^2x_i,z^{-2}y_i),[z^2t_i,z^{-2}u_i]\big).$$
The fixed locus of the action of $\GG_m$ on $E_T$ consists of the subspaces 
$$Z^+_T=\{u_1=\ldots=u_{s+1}=0\}=\PP^s\subseteq\PP^{n-1}=E_T,$$
$$Z^-_T=\{t_1=\ldots=t_{s+1}=0\}=\PP^s\subseteq\PP^{n-1}=E_T.$$

As $\Bl_0\AA^n$ is the total space $\VV(\cO_{E_T}(-1))$ of the line bundle $\cO_{E_T}(-1)=\cO_{E_T}(E_T)$ and the action of $\GG_m$  on $\Bl_0\AA^n$ coincides with the 
canonical action of $\GG_m$  on $\VV(\cO_{E_T}(-1))$ coming from the action of $G$ on $E_T=\PP^{n-1}$ given by
$$z\cdot[t_1,\ldots, t_{s+1}, u_1,\ldots, u_{s+1}]=[z^2t_1,\ldots, z^2t_{s+1}, z^{-2}u_1,\ldots, z^{-2}u_{s+1}],$$
it follows that $\cO_{E_T}(E_T)$ (and hence, $\cO(E_T)$) has a canonical $\GG_m$-linearization. With respect to this linearization, we have: 
\begin{equation}\label{B}
\wt_\la{\cO_{E_T}(-1)}_{|q}=\wt_\la{\cO(E_T)}_{|q}=+2,\quad q\in Z^+_T,\quad \la(z)=z,
\end{equation}
$$\wt_{\la}{\cO_{E_T}(-1)}_{|q}=\wt_{\la}{\cO(E_T)}_{|q}=-2,\quad q\in Z^-_T,\quad \la(z)=z.$$
and similarly, 
$$\wt_{\la'}{\cO_{E_T}(-1)}_{|q}=\wt_{\la'}{\cO(E_T)}_{|q}=-2,\quad q\in Z^+_T,\quad \la'(z)=z^{-1}.$$
$$\wt_{\la'}{\cO_{E_T}(-1)}_{|q}=\wt_{\la'}{\cO(E_T)}_{|q}=+2,\quad q\in Z^-_T,\quad \la'(z)=z^{-1}.$$

We denote by $\cO(\bar j)(\sum \al_T E_T)$ the line bundle 
$\pi^*\cO(j_1,\ldots, j_n)(\sum \al_T E_T)$ on $W_n$ (where $j_i, \al_T$ integers and $\pi: W_n\ra (\PP^1)^n$ is the blow-up map), 
with the $\GG_m$-linearization given 
by the tensor product of the canonical linearizations above.
As before, for every equivariant coherent sheaf 
$\cF$, we denote by $\cF\otimes z^k$ the tensor product with $\cO\otimes z^k$.
For a subset $E\subseteq N$, we denote 
$$\cO(-E):=\pi^*\cO(\bar j)$$ with $j_i=-1$ if $i\in E$ and $j_i=0$ otherwise.  
Note that the action of $S_2$ exchanges $\cO(-E)\otimes z^p$ with $\cO(-E)\otimes z^{-p}$ and $E_T$ with $E_{T^c}$ (Lemma \ref{dictionary2}).  

Consider the GIT quotient with respect to a (fractional) polarization 
$$\cL=\cO(1,\ldots,1)\left(-\eps\sum E_T\right),$$
where $0<\eps\ll 1$, $\eps\in\QQ$, and the sum is over all exceptional divisors (with the canonical polarization described above):
$$\tilde\Si_n=(W_n)_{ss}\sslash_{\cL}\GG_m.$$


\begin{lemma}
The $\GG_m$-linearized line bundle $\cO(\bar j)(\sum \al_T E_T)\otimes z^p$ descends to the GIT quotient $\tilde\Si_n$ if and only if
for all subsets $I\subseteq N$ with $|I|\neq s+1$
$$-\sum_{i\in I}j_i+\sum_{i\in I^c} j_i+p\quad \text{ is even }$$ and for all subsets $I\subseteq N$ with $|I|=s+1$, we have
$$-\sum_{i\in I}j_i+\sum_{i\in I^c} j_i+p\pm2\al_I\quad \text{ is divisible by }\quad 4.$$ 
\end{lemma}

\bp
By Kempf's descent lemma, a $G$-linearized line bundle $L$ descends to the GIT quotient if and only if the stabilizer of any point in the semistable 
locus acts trivially on the fiber of $L$ at that point, or equivalently, $\wt_\la L_{|q}=0$, for any semistable point $q$ and any $1$-PS
$\la:\GG_m\ra G$. By definition, $\wt_\la L_{|q}=\wt_\la L_{|p_0}$, where $p_0$ is the fixed point $\lim_{t\to 0}\la(t)\cdot q$. 

For any point $q$ in $(\PP^1)^n\setminus\{p_T\}$ such that $q=(z_i)$ has $z_i=\infty$ for $i\in I$ and 
$z_i\neq\infty$ for $i\in I^c$, we have for $\la(z)=z$ that $\lim_{t\to 0}\la(t)\cdot q$ is the point with coordinates
$z_i=\infty$ for $i\in I$ and $z_i=0$ for $i\in I^c$, and hence: 
\begin{equation}\label{A}
\wt_{\la}\big(\cO(\bar j)(\sum \al_T E_T)\otimes z^p\big)_{|q}=-\sum_{i\in I}j_i+\sum_{i\in I^c} j_i+p.
\end{equation}
Note that such a point $q$ is semistable if and only if $|I|<s+1$. 
Similarly, if $q$ has $z_i=0$ for $i\in I^c$ and 
$z_i\neq0$ for $i\in I$, $\la'(z)=z^{-1}$:
$$\wt_{\la'}\big(\cO(\bar j)(\sum \al_T E_T)\otimes z^p\big)_{|q}=\sum_{i\in I}j_i-\sum_{i\in I^c} j_i+p.$$
Note, $q$ is semistable iff $|I|>s+1$. The stabilizer of $q$ is $\{\pm1\}$ in both cases. 

If $q\in E_T\setminus (Z^+_T\sqcup Z^-_T)$ then $\lim_{t\to 0}\la(t)\cdot q\in Z^-_T$,
$\lim_{t\to 0}\la'(t)\cdot q\in Z^+_T$ and using (\ref{B}) we obtain
$$\wt_{\la}\big(\cO(\bar j)(\sum \al_T E_T)\otimes z^p\big)_{|q}=-\sum_{i\in I}j_i+\sum_{i\in I^c} j_i+p-2\al_T,$$
$$\wt_{\la'}\big(\cO(\bar j)(\sum \al_T E_T)\otimes z^p\big)_{|q}=\sum_{i\in I}j_i-\sum_{i\in I^c} j_i+p-2\al_T.$$
A point $q\in E_T\setminus (Z^+_T\sqcup Z^-_T)$ has stabilizer $\{\pm1, \pm i\}$. The conclusion follows. 
\ep

\begin{cor}
For $E\subseteq N$, $p\in\ZZ$, the line bundle $\cO(-E)(\sum\al_T E_T)\otimes z^p$ descends to the GIT quotient $\tilde\Si_n$ if and only if 
for all subsets $I\subseteq N$ with $|I|\neq s+1$
$$|I\cap E|-|I^c\cap E|+p\quad \text{ is even }$$ and for all subsets $I\subseteq N$ with $|I|=s+1$, we have
$$|I\cap E|-|I^c\cap E|+p-2\al_I\quad \text{ is divisible by }\quad 4.$$ 
\end{cor}

\begin{lemma}\label{identify2}
If $n=|N|$ is even, the Hassett space $Z_N=\M_{0,(\frac{1}{2}+\eta, \frac{1}{2}+\eta,\frac{1}{n},\ldots,\frac{1}{n})}$ 
is isomorphic to the GIT quotient $\tilde\Si_n=(W_n)_{ss}\sslash_{\cO(1,\ldots,1)(-\epsilon\sum E_T)}\bG_m.$ 
\end{lemma}

\bp
The trivial $\PP^1$-bundle $(\PP^1)^n\times \PP^1\ra (\PP^1)^n$ has sections $s_0$, $s_{\infty}$, $s_i$. We still denote by $s_0$, $s_{\infty}$, $s_i$ the induced sections of the pull back
$W_{ss}\times\PP^1\ra W_{ss}$.
The family is not $\cA$-stable at the points $p_T$, where $s_i=s_{\infty}$ for all $i\in T$ and $s_i=s_0$ for all $i\in T^c$
(markings in $T$ are identified with $\infty$, and markings in $T^c$ with $0$). 
Here
$\cA=(\frac{1}{2}+\eta, \frac{1}{2}+\eta,\frac{1}{n},\ldots,\frac{1}{n})$. Let $\cC'$ be the blow-up of $W\times \PP^1$ along the codimension 
$2$ loci 
$$E_T\times\{0\}=s_0(E_T),\quad E_T\times\{\infty\}=s_{\infty}(E_T).$$
Denote by $\tilde E^0_T$ and $\tilde E^{\infty}_T$ the corresponding exceptional divisors in $\cC'$. 
The resulting family $\pi': \cC'\ra W$ has fibers above points $p\in E_T$ a chain of $\PP^1$'s of the form $C_0\cup \tilde F\cup C_{\infty}$, where $\tilde F$ is the 
proper transform of the fiber of $W\times\PP^1\ra W$ and $\tilde F$  meets each of $C_0$ (the fiber of $\tilde E^0_T\ra E_T$ at $p$) and $C_{\infty}$ 
(the fiber of $\tilde E^{\infty}_T\ra E_T$ at $p$). The proper transforms of $s_i$ for $i\in T$ (resp., $i\in T^c$) intersect $C_{\infty}$ (resp., $C_0$) at distinct points. 
The dualizing sheaf $\om_{\pi'}$ is relatively nef, with degree $0$ on $\tilde F$. It follows that $\om_{\pi'}$ induces a morphism $\cC'\ra \cC$ over $W_{ss}$ which contracts the component $\tilde F$ in each of the above fibers, resulting in an $\cA$-stable family. Therefore, we have an induced morphism $F: W_{ss}\ra Z_N$. Clearly, the map 
$F$ is $\GG_m$-equivariant (where $\GG_m$ acts trivially on $Z_N$). As the GIT quotient $\tilde\Si_n$ is a categorical quotient,
there is an induced morphism $f: \tilde\Si_n\ra Z_N$. Two elements of the family $\cC\ra W_{ss}$ are isomorphic if and only if they belong to the same orbit under the action of $\GG_m$. Hence, the map $f$ is one-to-one on closed points (as there are no strictly semistable points in $W_{ss}$, $\tilde\Si_n$ is a good 
categorical quotient \cite[p. 94]{Dolgachev}). It follows that $f$ is an isomorphism. 
\ep

\begin{lemma}\label{dictionary2}
Assume $n=2s+2$ is even. We have the following dictionary between tautological line bundles on the Hassett space $Z_N$ (idenitified with  
the GIT quotient $\tilde\Si_n$) and $\GG_m$-linearized line bundles on $W_n$:
$$\cO(\de_{i0})=pr^*_i\cO(1)\left(-\sum_{i\notin T}E_T\right)\otimes z,\quad \cO(\de_{i\infty})=pr^*_i\cO(1)\left(-\sum_{i\in T}E_T\right)\otimes z^{-1}$$
$$\psi_0=\cO\left(\sum E_T\right)\otimes z^{-2},\quad \psi_{\infty}=\cO\left(\sum E_T\right)\otimes z^{2},\quad \psi_i=pr^*_i\cO(-2)\left(\sum E_T\right)\otimes 1,$$
$$\cO(\de_{T\cup\{\infty\}})=\cO(2E_T)\otimes 1\quad (|T|=s+1).$$
\end{lemma}

\bp
Denote $\de_T=\de_{T\cup\{\infty\}}$. 
We start with the proof of $\cO(\de_T)=\cO(2E_T)\otimes 1$. 
Consider the affine chart $$\AA^n=(\PP^1\setminus\{\infty\})^{s+1}\times (\PP^1\setminus\{0\})^{s+1}$$ 
around the point $p_T$ (markings in $T=\{s+2,\ldots,n\}$ are identified with $\infty$, and markings in $T^c$ with $0$). 
We have coordinates $x_1,\ldots, x_{s+1}, y_1,\ldots, y_{s+1}$. 
The GIT quotient map $(\PP^1)^n_{ss}\ra \Si$ is locally at $p_T$ given by
$$f: \AA^n\ra Y=f(\AA^n)\subseteq\AA^{(s+1)^2},\quad f((x_i),(y_j))=(x_iy_j)_{ij}.$$
The morphism $F: W_{ss}\ra \tilde\Si_n=\tilde\Si$ induced by the universal family over $W_{ss}$ (proof of Lemma \ref{identify2}) is locally the 
restriction to the semistable locus of the rational map (which we still call $F$)
$$F: \Bl_0\AA^n\dra\Bl_0 Y\subseteq\Bl_0\AA^{(s+1)^2}.$$
Consider coordinates $((x_i,y_i),[t_i,u_i])$ (with $x_it_j=x_jt_i$, $x_iu_j=y_jt_i$, $x_it_j=x_jt_i$) on $\Bl_0\AA^n\subseteq \AA^n\times\PP^{n-1}$ and coordinates 
$(z_{ij},[w_{ij}])$ on $\Bl_0\AA^{(s+1)^2}$ (with $z_{ij}w_{kl}=z_{kl}w_{ij}$). Consider the affine charts $U_1=\{t_1\neq0\}\subseteq \Bl_0\AA^n$  and 
$V_{1j}=\{w_{1j}\neq0\}\subseteq \Bl_0\AA^{r^2}$. The map $F_{|U_1}$ is the rational map
$$F: U_1=\AA^n_{x_1,t_2,\ldots,t_r,u_1,\ldots,u_r}\dra V_{1j}=\AA^{r^2}_{z_{1j},(w_{kl})_{kl\neq 1j}},$$
$$z_{1j}=x_1^2u_j,\quad w_{kl}=\frac{t_ku_l}{u_j}.$$
The exceptional divisor $\tilde E$ in $\Bl_0\AA^{(s+1)^2}$ has local equation $z_{1j}=0$ in $V_{1j}$, while the exceptional divisor $E_T$ of $\Bl_0\AA^n$ has equation
$x_1=0$ in $U_1$. It follows that $F^*\cO(\tilde E)=\cO(2E_T)$. In particular, as $\de_T=\Bl_0 Y\cap \tilde E$, it follows that 
$F^*\cO(\de_T)=\cO(2E_T)$. It follows that $\cO(\de_T)=\cO(2E_T)\otimes z^k$, for some integer $k$ 
(the same for all $T$, by the $S_n$-symmetry). On the other hand, by the $S_2$-symmetry,  $\cO(\de_{T^c})=\cO(2E_{T^c})\otimes z^{-k}$. Hence, we must have $k=0$. 

We now prove that $\cO(\de_{i0})=pr^*_i\cO(1)(-\sum_{i\notin T}E_T)\otimes z$. (Note that all other relations will then follow by $S_2$-symmetry and 
Lemma \ref{relations2}.) Clearly, $F^*\cO(\de_{i0})$ is the line bundle $\cO(\tilde\De_{i0})_{|W_{ss}}$, where $\tilde\De_{i0}$ is the 
proper transform in $W$ of the diagonal $\De_{i0}$ in $(\PP^1)^n$ defined by $x_i=0$, where $z_i=[x_i,y_i]$ now denote coordinates on 
$(\PP^1)^n$. As $\tilde\De_{i0}=\De_{i0}-\sum_{i\notin T}E_T$ (markings in $T^c$ are identified with $0$), it follows that 
$$\cO(\de_{i0})=pr^*_i\cO(1)\left(-\sum_{i\notin T}E_T\right)\otimes z^k,$$
for some integer $k$. The pull-back of the canonical section of the effective divisor $\de_{i0}$ (which is $x_i$) must be an invariant section.  
The section $x_i$ of $\cO_{\PP^1}(1)$ becomes the constant section $1$ in the open chart $U:~x_i\neq 0$. Considering a point 
$q=(q_1,\ldots, q_n)$ in $U$, with $q_i=\infty$ and $q_j\in\PP^1$ general for $j\neq i$, it follows that 
for the $1$-PS $\la(z)=z$ we have $\wt_{\la} pr_i^*\cO(1)_{|q}=-1$, $\wt_{\la} \cO\otimes z^k_{|q}=k$, hence, the constant section $1$ becomes 
$z^{-1+k}$ under the action of $\la$ and we must have $k=1$ for the section to be invariant. 
\ep

\begin{lemma}\label{restrict to delta}
Let $\de_T:=\de_{T\cup\{\infty\}}=\PP^s\times \PP^s$. We have 
$${\de_{i\infty}}_{|\de_T}=
\begin{cases}
\cO(1,0) & \text{ if }\quad i\in T\cr
\cO & \text{ if }\quad i\notin T\end{cases},\quad
{\de_{i0}}_{|\de_T}=
\begin{cases}
\cO(0,1) & \text{ if }\quad i\notin T\cr
\cO & \text{ if }\quad i\in T,\end{cases},$$
$${\psi_{\infty}}_{|\de_T}=\cO(-1,0),\quad {\psi_{0}}_{|\de_T}=\cO(0,-1),\quad \de_T{|\de_T}=\cO(-1,-1).$$
\end{lemma}

\bp
By symmetry, it suffices to compute ${\de_{i\infty}}_{|\de_T}$ and ${\psi_{\infty}}_{|\de_T}$. 
Clearly, the intersection ${\de_{i\infty}}\cap \de_T=\emptyset$ if  $i\notin T$. 
We identify $\de_T=\M'\times \M''=\PP^s\times \PP^s$, 
where $\M'$, resp., $\M''$ are Hassett spaces with weights 
$(\frac{1}{2}+\eta, \frac{1}{n},\ldots, \frac{1}{n}, 1)$, with the attaching point $x$ having weight $1$. 
We identify $\M'=\PP^s$ via the isomorphism $|\psi_x|: \M'\ra\PP^s$. 
We have ${\de_{i\infty}}_{|\de_T}=\de_{i\infty}\otimes\cO$, ${\psi_{\infty}}_{|\de_T}=\psi_{\infty}\otimes\cO$. 
By Lemma \ref{relations}, on $\M'$ we have $\psi_{\infty}+\psi_x=0$ since $\de_{x\infty}=0$,
and $\de_{i\infty}=-\psi_{\infty}=\cO(1)$ if $i\in T$. 
The identity $\de_T{|\de_T}=\cO(-1,-1)$ follows now from the previous ones by restricting to $\de_T$ any of the identities in Lemma \ref{relations2}. 
\ep

\subsection{Exceptionality}
Note that $W_n$ is a polarized toric variety with the polytope $\Delta$
obtained by truncating the $n$-dimensional cube at vertices lying on the hyperplane $H$ normal to and bisecting the big diagonal.
Then $\tilde\Si_n$ is a smooth polarized projective  toric variety for the torus $\bG_m^{n-1}$ and  its polytope is $\Delta\cap H$.
In particular, the topological Euler characteristic $e(\tilde\Si_n)$ is equal to the number of 
edges $\Delta$ intersecting $H$:
$$e(\tilde Z_{n})=(s+1)^2{n\choose s+1}
=s^2{n\choose s+1}+(n-1){n\choose s+1}\quad (n=2s+2).
$$
Note that 
$(s+1){n\choose s+1}=n{n\choose 0}+(n-2){n\choose 1}+(n-4){n\choose 2}+\ldots+2{n\choose s}$. 

\begin{defn}\label{even translate}
For $E\subseteq N$, $e=|E|$, $p\in\ZZ$ such that $p+e$ is even, let 
$$L_{E,p}:=\cO(-E)\left(\sum_{T\subseteq N, |T|=s+1}\al_{T,E,p} E_T\right)\otimes z^p\quad \text{where}$$
\begin{equation}\label{The alpha}
\al_{T,E,p}:=-|x_{T,E,p}|,\quad x_{T,E,p}:=|E\cap T|-\frac{e-p}{2}
\end{equation}
i.e., the descent to $\tilde \Si_n$ of the restriction to $(W_n)_{ss}$ of
the above $\GG_m$-linearized line bundle on $W_n$. By Lemma \ref{dictionary2} we recover Definition \ref{L even case}:
\begin{equation}\label{again}
L_{E,p}=-\left(\frac{e-p}{2}\right)\psi_{\infty}-\sum_{i\in E}\de_{i\infty}-\sum_{x_{T,E,p}>0} x_{T,E,p}\de_{T\cup\{\infty\}}.
\end{equation}
We write $x_T$ if there is no ambiguity. Note that $x_{T,E,p}=-x_{T^c,E,-p}$.

\end{defn}

\begin{lemma}
The action of $S_2$ on $Z_N$ exchanges $L_{E,p}$ with $L_{E,-p}$.
\end{lemma}

\bp
The statement follows immediately from (\ref{again}) and Lemma \ref{relations2}.  
\ep

\bp[Proof of Theorem \ref{odd} - exceptionality]
Lemma \ref{Torsion} implies that the torsion sheaves $\cO_{\de}(-a,-b)$ form an exceptional collection.
Let now $\de:=\de_{T\cup\{\infty\}}$. 
To prove that $\{\cO_{\de}(-a,-b), L_{E,p}\}$ form an exceptional pair, i.e., 
that $L^\vee_{|\de}\otimes\cO(-a,-b)$ is acylic, note that 
by Lemma \ref{restrict to delta} and (\ref{even translate}) we have, letting $\al_T:=\al_{T,E,p}$:
$$L^\vee_{|\de}=\begin{cases}
\cO(0,\al_T) & \text{ if }\quad p+|E\cap T|-|E\cap T^c|\geq0\cr
\cO(\al_T,0) & \text{ if }\quad p+|E\cap T|-|E\cap T^c|\leq0,\end{cases}.$$
Clearly, if  $a,b>0$ then $L^\vee_{|\de}\otimes\cO(-a,-b)$ is acylic. Consider now the case when one of $a,b$ is $0$. 
Using the $S_2$-symmetry, we may assume $a=0$. Let $0<b<\frac{s+1}{2}$.
Since by (\ref{alpha ineq}) we have $-\lfloor\frac{s+1}{2}\rfloor\leq \al_T\leq0$, the result follows. 

\smallskip

We describe the Kempf-Ness stratification of the unstable locus in $W_n$.
Let $G=\GG_m$. As before, we consider $(\la, Z)$, with 
a $1$-PS $\la: \GG_m\ra G$ and $Z$ a connected component of the $\la$-fixed locus. 
It suffices to consider $\la(z)=z$ and $\la'(z)=z^{-1}$. 
The $G$-fixed locus in $W=W_n$ consists of the points 
$$Z_I=\{(x_i)\,|\,x_i=\infty\ \hbox{\rm for}\ i\in I,\
x_i=0\ \hbox{\rm for}\ i\not\in I,\}\in(\PP^1)^n\setminus\{p_T\}$$
for every subset $I\subseteq N$ with $|I|\neq s+1$ and the loci
$Z^+_T\sqcup Z^-_T\subseteq E_T$, for each subset $T\subseteq N$, $|T|=s+1$. The pairs 
$(\la, Z)$ to be considered are therefore 
$$(\la, Z_I),\quad (\la', Z_I)\quad (I\subseteq N, |I|\neq s+1),$$
$$(\la, Z^+_T),\quad (\la', Z^+_T),\quad (\la, Z^-_T),\quad (\la', Z^-_T)\quad (T\subseteq N, |T|=s+1).$$
Recall that our polarization  is $\cL=\cO(1,\ldots,1)(-\epsilon\sum E_T)$ and 
for any subset $I\subseteq N$ with $|I|\neq s+1$ we have
$$\hbox{\rm weight}_\lambda\cL|_{Z_I}=|I^c|-|I|, \quad \hbox{\rm weight}_{\lambda'}\cL|_{Z_I}=-|I^c|+|I|,$$
while for all subsets $T\subseteq N$ with $|T|=s+1$ we have:
$$\hbox{\rm weight}_\lambda\cL|_{q}=-2\epsilon,  \quad \hbox{\rm weight}_{\lambda'}\cL|_{q}=+2\epsilon\quad (q\in Z^+_T),$$
$$\hbox{\rm weight}_\lambda\cL|_{q}=+2\epsilon,  \quad \hbox{\rm weight}_{\lambda'}\cL|_{q}=-2\epsilon\quad (q\in Z^-_T).$$
As in the $n$ odd case, we define for any subset $I\subseteq N$ affine subsets: 
$$S_I=\{(x_i)\,|\,x_i=\infty\ \hbox{\rm  if}\ i\in I, x_i\neq\infty\ \hbox{\rm  if}\ i\notin I\}\cong \AA^{|I^c|}$$
$$S'_I=\{(x_i)\,|\,x_i=0\ \hbox{\rm if}\ i\not\in I,  x_i\neq0\ \hbox{\rm  if}\ i\in I\}\cong \AA^{|I|}.$$
The unstable locus arises from the pairs with negative weight:
$$(\la, Z_I)\quad (\text{for } |I|>s+1),\quad (\la', Z_I)\quad (\text{for } |I|<s+1),$$
$$(\la, Z^+_T),\quad (\la', Z^-_T)\quad (\text{for } |T|=s+1):$$ 
$$S_I\cong \AA^{|I^c|}\quad (\text{for }  |I|>r),\quad S'_I\cong \AA^{|I|}\quad (\text{for }  |I|<s+1),$$
$$S^+_T=\Bl_{p_T}S_T=\Bl_0\AA^{|T^c|},\quad S^-_T=\Bl_{p_T}S'_T=\Bl_0\AA^{|T|}\quad (\text{for } |T|=s+1).$$

The destabilizing $1$-PS for $S_I$ (resp.,~for $S'_I$) is $\lambda$ 
(resp.~$\lambda'$). The $1$-PS $\lambda$ (resp., $\lambda'$) acts on the restriction to $Z_I$ of the conormal bundle 
$N^\vee_{S_I|(\PP^1)^n}$ (resp., $N^\vee_{S'_I|(\PP^1)^n}$) with positive weights. Their sum $\eta_I$ (resp., $\eta'_I$) is:
$$\eta_I=2|I|,\quad\hbox{\rm resp.,}\quad \eta'_I=2|I^c|.$$

When $|T|=s+1$, the destabilizing $1$-PS for $S^+_T$ (resp.~for $S^-_T$) is $\lambda$ 
(resp.~$\lambda'$). The $1$-PS $\lambda$ (resp., $\lambda'$) acts on 
$N^\vee_{S^+_T|W}$ (resp., $N^\vee_{S^-_T|W}$) restricted to $q\in Z^+_T$ (resp., $Z^-_T$ ), with positive weights.  
Their sum $\eta^+_T$ (resp., $\eta^-_T$) is:
$$\eta^+_T=4|T|=2n,\quad\hbox{\rm resp.,}\quad \eta^-_T=4|T^c|=2n.$$
To see this, let $q\in Z^+_T$. The sum of $\lambda$-weights of $\big(N^\vee_{S^+_T|W}\big)_{|q}$  equals 
$$\hbox{\rm weight}_\lambda \big(\det N^\vee_{S^+_T|W}\big)_{|q}=\hbox{\rm weight}_\lambda \big(\det T_{S^+_T}\big)_{|q}-\hbox{\rm weight}_\lambda \big(\det T_W\big)_{|q}.$$
We use the local coordinates introduced in \ref{even intro} (assume again w.l.o.g. that $T=\{s+2,\ldots, n\}$). 
We may assume also that the point 
$$q=[t_1,\ldots, t_{s+1},0\ldots,0]\in Z^+_T\subseteq E_T=\PP^{n-1}$$ has $t_1=1$. 
Then local coordinates on an open set $U=\AA^n\subseteq W$ around $q$ are given by
$x_1, t_2,\ldots, t_{s+1}$, $u_1,\ldots, u_{s+1},$
with the blow-up map  $\AA^n\ra \AA^n$:
$$(x_1, t_2,\ldots, t_{s+1}, u_1,\ldots, u_{s+1})\mapsto (x_1, x_1t_2,\ldots, x_1 t_{s+1},  x_1u_1,\ldots, x_1 u_{s+1}).$$
Then $S^+_T\cap U\subseteq U$ has equations $u_1=\ldots=u_{s+1}$ (the proper transform of 
$S_T: y_1=\ldots=y_{s+1}$). The action of $G$ on $W$ induces an action on $U$:
$$z\cdot x_1=z^2 x_1,\quad z\cdot t_i=t_i\quad (i=2,\ldots,s+1), z\cdot t_i=z^{-4}t_i \quad (i=1,\ldots,s+1).$$
It follows that 
$$\hbox{\rm weight}_\lambda \big(\det T_{W}\big)_{|q}=-2-4s,\quad \hbox{\rm weight}_\lambda \big(\det T_{S^+_T}\big)_{|q}=2.$$
Hence, $\eta_T=4s+4=2n$. Similarly, $\eta'_T=2n$: for $q\in Z^-_T$ and
coordinates $y_1, t_1,\ldots, t_{s+1}, u_2,\ldots, u_{s+1}$ on the chart $u_1=1$, the action of $G$ given by: 
$$z\cdot y_1=z^{-2} y_1,\quad z\cdot t_i=z^4t_i\quad (i=1,\ldots,s+1), z\cdot u_i=u_i \quad (i=2,\ldots,s+1).$$ 
Letting $m:=\Bigl\lfloor\frac{n}{4}\Bigr\rfloor=\Bigl\lfloor\frac{s+1}{2}\Bigr\rfloor$, we make a choice of windows $\GG_w$:
$$[w_I, w_I+\eta_I),\quad  [w'_I, w'_I+\eta'_I),\quad [w^+_T, w^+_T+\eta^+_T),\quad  [w^-_T, w^-_T+\eta^-_T),$$ 
$$w_I=w'_I=-(s+1),\quad w^+_T=w^-_T=-4m=-n\quad \text{ if}\quad s\quad \text{is odd},$$
$$w_I=w'_I=-s,\quad w^+_T=w^-_T=-4m=-n+2\quad \text{ if}\quad s\quad \text{is even}.$$

We prove that $\GG_w$ contains the $G$-linearized line bundles that descend to the 
$L_{E,p}$ in Theorem \ref{even}. Since the collection is $S_2$ invariant and $S_2$ flips the strata $S_I$ and $S'_I$,
it suffices to check the window conditions for the strata $S_I$,  $S^+_T$.  
For $I\subseteq N$, $|I|>s+1$, at the point $Z_I\in S_I$ we have by (\ref{A})
$$\hbox{\rm weight}_\lambda\big(L_{E,p}\big)_{|Z_I}=|E\cap I|-|E\cap I^c|+p,$$
which lies in $[w_I, w_I+\eta_I)$ by Lemma \ref{MaxMin}. 

For $T\subseteq N$ with $|T|=s+1$, we have by   (\ref{B}) and  (\ref{A}) that
\begin{align*}
\hbox{\rm weight}_\lambda\big(L_{E,p}\big)_{|q\in Z^+_T}&=|E\cap T|-|E\cap T^c|+p-2|x_T|\\
&=\begin{cases}
0 & \text{ if }\quad |E\cap T|-|E\cap T^c|+p\geq0\cr
-4|x_T| & \text{ if }\quad |E\cap T|-|E\cap T^c|+p\leq0,\end{cases}
\end{align*}
which by (\ref{alpha ineq}) lies in $[w^+_T, w^+_T+\eta^+_T)$.
Hence, all $\{L_{E,p}\}$ in Theorem \ref{even} are contained in the window $\GG_w$. 

We now check exceptionality. Consider two line bundles as in Theorem \ref{even}:
$$L_{E,p}=\cO(-E)\left(\sum_{|T|=r}\al_T E_T\right)\otimes z^p,\quad L_{E',p'}=\cO(-E')\left(\sum_{|T|=r}\al'_T E_T\right)\otimes z^{p'}.$$  
where $\al_T:=\al_{T,E,p}$, $\al'_T:=\al_{T,E',p'}$.  Assume that $e=|E|\geq e'=|E'|$. Hence, $E\nsubseteq E'$ unless $E=E'$. 
By the main result of \cite[Theorem 2.10]{DHL}, we have that $R\Hom(L_{E', p'}, L_{E,p})$ equals the 
weight $(p'-p)$ part (with respect to the canonical action of $G$) of 
$$R\Hom_{W}(L_{E',p'},L_{E,p})=R\Ga\left(\cO(E'-E)\otimes\cO(\sum_{T}(\al'_T-\al_T)E_T)\right).$$ 
Hence, letting
$$M_0:=\cO(E'-E)\otimes\cO\left(\sum_{\be_T\leq0}(-\be_T )E_T\right),\quad \text{where}\quad \be_T:=\al_T-\al'_T,$$
we need to understand the weight $(p'-p)$ part of 
$$R\Ga\left(M_0\otimes\cO\left(\sum_{|T|=r,\be_T>0}(-\be_T )E_T\right)\right).$$
Note that $M_0$ is a pull-back from $(\PP^1)^n$; hence, 
by the projection formula, $R\Ga(M_0)=R\Ga(\cO(E'-E))$ (which is $0$ if $E\nsubseteq E'$). 

Consider a simplified situation. For a line bundle $M$ on $W$,  $G:=E_T$, $\be:=\be_T>0$ consider the exact sequences:
$$0\ra M(-(i+1) G)\ra M(-iG)\ra M(-iG)_{|G}\ra 0,\quad (i=0,\ldots \be-1).$$
To prove that the weight $(p'-p)$ of  $R\Ga(M(-\be G))$ is $0$, it suffices to prove that
$$R\Ga(M),\quad R\Ga(M(-iG)_{|G})\quad (i=0,1,\ldots,\be-1),$$ 
have no weight $(p'-p)$ part. Put an arbitrary order on the subsets $T$ with $\be_T>0$ 
($T_1,T_2,\ldots$).  
Applying the above observation successively, first for $M_0$, $E_{T_1}$, then inductively 
for $M_0(-\be_1T_1-\ldots-\be_iT_i )$, $E_{T_{i+1}}$, it suffices to prove that for all $T$, the following spaces  
$$R\Ga(M_0),\quad R\Ga(M_0(-iE_T)_{|E_T})\quad (i=0,1,\ldots,\be-1)$$ 
have no weight $(p'-p)$ part. 

We start with $R\Ga(M_0)$. 
If $E\neq E'$, then $R\Ga^*(M_0)=0$. If $E=E'$, then $M_0=\cO$ and the action of $G$ on $R\Ga(M_0)$ is trivial. Hence,
unless $p=p'$ (i.e., $L_{E,p}=L_{E',p}$), $R\Ga(M_0)$ has no weight $(p'-p)$ part. 

We now continue with $R\Ga(M_0(-iE_T)_{|E_T})$. By the projection formula, 
$$R\Ga(M_0(-iE_T)_{|E_T})={M_0}_{|p_T}\otimes R\Ga(\cO(-iE_T)_{|E_T}),$$
where ${M_0}_{|p_T}$ is the fiber of $M_0$ at $p_T$ (we denote $M_0$ both the line bundle on $(\PP^1)^n$ 
and its pull back to $W$). By (\ref{A}), the action of $G$ on ${M_0}_{|p_T}$ has weight 
$$\big(|E\cap T|-|E\cap T^c|\big)-\big(|E'\cap T|-|E'\cap T^c|\big).$$
Consider coordinates $t_i, u_i$ on $E=\PP^{n-1}$, such that $t_i$ (resp., $u_i$) have weight $2$ (resp., weight $-2$). 
There is a canonical identification 
$$R\Ga(\cO(-iE_T)_{|E_T})=\CC\left\{\prod t_k^{a_k}\prod u_k^{b_k}\quad |\quad a_k,b_k\in\ZZ_{\geq0},\quad \sum a_k+\sum b_k=i\right\},$$ 
with the weight of $\prod t_k^{a_k}\prod u_k^{b_k}$ equal to $2\sum a_k-2\sum b_k$. As $2\sum a_k-2\sum b_k$ ranges through all even numbers between
$-2i$ and $2i$, it follows that the possible weights of elements in $R\Ga(M_0(-iE_T)_{|E_T})$ are 
$$\big(|E\cap T|-|E\cap T^c|\big)-\big(|E'\cap T|-|E'\cap T^c|\big)+2j,$$
for all the values of $j$ between $-i$ and $i$.  

Assume now that for some $0\leq i\leq \be_T-1=\al_T-\al'_T-1$, $-i\leq j\leq i$,
$$\big(|E\cap T|-|E\cap T^c|\big)-\big(|E'\cap T|-|E'\cap T^c|\big)+2j=p'-p.$$
Using the definition of $\al_T$, $\al_{T'}$, it follows that $\pm 2\al_T \pm 2\al'_T=-2j$. 
\begin{claim}
None of $\pm \al_T \pm \al'_T$ lies in the interval 
$[-(\al_T-\al'_T-1), (\al_T-\al'_T-1)]$. 
\end{claim}
\bp
By symmetry, it is enough to prove that none of $\pm \al_T \pm \al'_T$ lies in the interval 
$[0, (\al_T-\al'_T-1)]$. As $\al_T, \al'_T\leq0$ and $\al_T>\al'_T$. 
Hence, it remains to prove that $-\al_T-\al'_T$, $\al_T-\al'_T$ do not lie in the interval
$[0, (\al_T-\al'_T-1)]$. But clearly, $-\al_T-\al'_T>\al_T-\al'_T-1$ and $\al_T-\al'_T>\al_T-\al'_T-1$. 
\ep
This finishes the proof that the collection in Theorem \ref{even} is exceptional. 
\ep

\begin{lemma}\label{Torsion}
Let $0\leq a, b\leq s$. Let $\de$ be a divisor in a Hassett space $\M$ such that 
$\de=\PP^s\times\PP^s$ and with normal bundle $\cO(-1,-1)$. Assume that 
the restriction map $\Pic(\M)\ra\Pic(\de)$ is surjective. 
Then 
$\{\cO_{\de}(-a,-b), \cO_{\de}(-a',-b')\}$
is not an exceptional collection if and only if one of the following happens:
\bi
\item $a'\geq a$, $b'\geq b$,
\item $a'=0$, $a=s$, $b'>b$, 
\item $b'=0$, $b=s$, $a'>a$, 
\item $a'=b'=0$, $a=b=s$.
\ei
When $a=a'$, $b=b'$, we have $R\Hom(\cO_{\de}(-a',-b'), \cO_{\de}(-a,-b))=\CC$. 
\end{lemma}

\bp
As any line bundle on $\de$ is the restriction of a line bundle on $\M$, we have that 
$$R\Hom(\cO_{\de}(-a',-b'), \cO_{\de}(-a,-b))=R\Hom(\cO_{\de}, \cO_{\de}(a'-a,b'-b)).$$
Applying $R\Hom(-, \cO_{\de}(a'-a,b'-b))$ to the canonical sequence 
$$0\ra\cO(-\de)\ra\cO\ra\cO_{\de}\ra0,$$
it follows that there is a long exact sequence on $\M$
$$\ldots\ra\Ext^i(\cO_{\de}, \cO_{\de}(a'-a,b'-b))\ra$$
$$\ra\HH^i(\cO_{\de}(a'-a,b'-b))\ra \HH^i(\cO_{\de}(a'-a-1,b'-b-1))\ra\ldots$$
It is clear now that if any of the conditions in the Lemma hold, then 
$$R\Hom(\cO_{\de}(-a',-b'), \cO_{\de}(-a,-b))\neq0.$$
Assume now that none of the conditions holds.  Then either $a'<a$ or $b'<b$. Assume $a'<a$. Since $a'-a\geq -a\geq -s$, 
$\cO_{\de}(a'-a,b'-b)$ is acyclic. But in this case $\cO_{\de}(a'-a-1,b'-b-1)$ is not acyclic if and only if $a'=0$, $a=s$ and either
$b'-b>0$ or $b'-b\leq -s$ (in which case, we must have $b'=0$, $b=s$). This gives precisely two of the listed cases. The case   
$b'<b$ is similar. 
\ep

\begin{lemma}\label{MaxMin}
Let $n=2s+2$. For a fixed set $I\subseteq N$ with $|I|>s+1$, we have
$$\max_{(E,p)}\big(|E\cap I|-|E\cap I^c|+p\big)=
\begin{cases}
2|I| -(s+3) & \text{ if}\quad s\quad \text{is odd} \cr
2|I| -(s+2) & \text{ if}\quad  s\quad \text{is even},\end{cases}$$
$$\min_{(E,p)}\big(|E\cap I|-|E\cap I^c|+p\big)=
\begin{cases}
-(s+1) & \text{ if }\quad  s\quad \text{is odd}\cr
-s & \text{ if }\quad s\quad \text{is even},\end{cases}$$
where the maximum and the minimum are taken over all the pairs $(E,p)$ corresponding to each line bundle $L_{E,p}$ in Theorem \ref{even}. 
Similarly, for $T\subseteq N$, $|T|=s+1$
$$\max_{(E,p)}\big(p+|E\cap T|-|E\cap T^c|\big)=2m,\quad\min_{(E,p)}\big(p+|E\cap T|-|E\cap T^c|\big)=-2m,$$
$$\text{where}\quad m:=\Bigl\lfloor\frac{n}{4}\Bigr\rfloor=\Bigl\lfloor\frac{s+1}{2}\Bigr\rfloor.$$
In particular, when $(E,p)$ are as in Theorem \ref{even}, 
the coefficients $\al_{T,E,p}$ in (\ref{The alpha}) satisfy 
\begin{equation}\label{alpha ineq}
-m\leq \al_{T,E,p}=-|x_{T,E,p}|\leq 0
\end{equation}
\end{lemma}
The proof is straightforward and we omit it. 

\subsection{Fullness}
Let $\cC$ be the collection in  Theorem \ref{even}. We denote by $\cA\subset \cC$ the collection of torsion sheaves in Theorem \ref{even}. 
We prove more generally: 
\begin{thm}\label{full even}
The collection $\cC$ in Theorem \ref{even} generates all line bundles $\{L_{E,p}\}$ (see Definition \ref{L even case} and Definition \ref{even translate})
for all $E\subseteq N$, $e=|E|$, $p\in\ZZ$ with $e+p$ even. 
\end{thm}

\bp[Proof of Theorem \ref{even} - fullness]
By Theorem \ref{full even}, the collection $\cC$ generates all the objects $Rp_*({\pi^*_I}\hat\GG)$ from Corollary \ref{rewrite}. Fullness then follows by Corollary \ref{S is enough}. 
\ep

To prove Theorem \ref{full even} we do an induction on the \emph{score} $S(E,p)$:
\begin{equation}\label{score notn}
S(E,p):=|p|+\min\{e,n-e\},
\end{equation}
\begin{equation}\label{q}
\text{written as}\quad S(E,p)=2\Bigl\lfloor\frac{s}{2}\Bigr\rfloor+2q,\quad q\in\ZZ.
\end{equation}

\begin{rmk}\label{reformulate}
As $S(E,p)$ is even, the range of $(E,p)$ in Theorem \ref{even} is precisely:
\bi
\item If $s$ is even: $S(E,p)\leq s$,
\item If $s$ is odd:  $S(E,p)\leq s+1$ if $e\leq s+1$ and $S(E,p)\leq s-1$ if $e\geq s+2$.  
\ei

Using notation (\ref{q}), $(E,p)$ is not in the range of Theorem \ref{even} if $q\geq1$ when $s$ is even or $s$ is odd and $e\geq s+2$, and if $q\geq2$ when $s$ is odd and $e\leq s+1$. 
\end{rmk}

To prove Theorem \ref{full even} we  introduce three other types of line bundles. 
\begin{notn}\label{extraLB}
Let $n=2s+2$, $E\subseteq N$, $e=|E|$ and $p\in\ZZ$. 
On $Z_N$ let 
$$R_{E,p}=-\left(\frac{e-p}{2}\right)\psi_{\infty}-\sum_{i\in E}\de_{i\infty},\quad
Q_{E,p}=-\left(\frac{e+p}{2}\right)\psi_0-\sum_{i\in E}\de_{i0},$$
\begin{equation}\label{RS to V}
V_{E,p}:=R_{E,p}+\sum_{x_{T,E,p}<0}|x_{T,E,p}| \de_{T\cup\{\infty\}} 
=Q_{E,p}+\sum_{x_{T,E,p}>0}|x_{T,E,p}| \de_{T\cup\{\infty\}}, 
\end{equation}
where the last equality follows from (\ref{L to R}) and (\ref{L to S}). 
\end{notn}
We recall for the reader's convenience that using Notation \ref{The alpha} we have
$$L_{E,p}=-\left(\frac{e-p}{2}\right)\psi_{\infty}-\sum_{i\in E}\de_{i\infty}-\sum_{x_{T,E,p}>0}x_{T,E,p}\de_{T\cup\{\infty\}}.$$
Therefore, 
\begin{equation}\label{L to R}
R_{E,p}=L_{E,p}+\sum_{x_{T,E,p}>0} |x_{T,E,p}|\de_{T\cup\{\infty\}} 
\end{equation}
and by using Lemma \ref{relations2}, we have also 
\begin{equation}\label{L to S}
Q_{E,p}=L_{E,p}+\sum_{x_{T,E,p}<0} |x_{T,E,p}|\de_{T\cup\{\infty\}}. 
\end{equation}

We remark that using Lemma \ref{dictionary2}, we have:
$$R_{E,p}=\cO(-E)(\sum x_T E_T)\otimes z^p,\quad Q_{E,p}=\cO(-E)(-\sum x_T E_T)\otimes z^p,$$
$$L_{E,p}=\cO(-E)(-\sum |x_T| E_T)\otimes z^p,\quad V_{E,p}=\cO(-E)(\sum |x_T| E_T)\otimes z^p.$$

\begin{rmk}\label{symmetry}
It is clear by the definition that by the $S_2$ symmetry (i.e., exchanging $0$ with $\infty$) the line bundle 
$R_{E,p}$ is exchanged with $Q_{E,-p}$. The line bundles $R_{E,p}$, $Q_{E,p}$ will be crucial for the proof of Theorem \ref{full even}. We note that the line bundles $V_{E,p}$ are used only in the proof of Corollary \ref{next torsion}. 
\end{rmk}

For every divisor $\de_T:= \de_{T\cup\{\infty\}}$, we have by Lemma \ref{restrict to delta} that
\begin{equation}\label{resRS}
{R_{E,p}}_{|\de_T}=\cO(-x_{T,E,p},0),\quad {Q_{E,p}}_{|\de_T}=\cO(0, x_{T,E,p}).
\end{equation}
From here on, the notation $\cO(-a,-b)$ indicates that $\cO(-a)$ (resp., $\cO(-b)$) corresponds to the component marked by $\infty$
(resp., marked by $0$). 

\begin{defn}
We say that line bundles $L$ and $L'$ are \emph{related by quotients} $Q^i$ if there are exact sequences
$$0\ra L^{i-1}\ra L^i\ra Q^i\ra0\quad (i=1,\ldots, t),$$
$$L_0=L,\quad L_t=L'.$$ 
Note that when $L'=L+\sum \be_T \de_T$ with $\be_T\geq 0$ for all $T$, the quotients $Q^i$ are direct sums of torsion sheaves of type 
$\cO_{\de_T}(-a,-b)$. 
\end{defn}

\begin{lemma}\label{quotients}
Let $E\subseteq N$, $e=|E|$, $p\in\ZZ$, such that $e+p$ even. Then:
\bi
\item[(i) ]  $L_{E,p}$ and $R_{E,p}$ are related by quotients which are direct sums of type
$$\cO_{\de_T}(-x_T+i,i),\quad 0\leq i<|x_T|=x_T\quad (x_T>0)$$
\item[(ii) ] $L_{E,p}$ and $Q_{E,p}$ are related by quotients which are direct sums of type
$$\cO_{\de_T}(i,x_T+i),\quad 0\leq i<|x_T|=-x_T\quad (x_T<0)$$
\item[(iii) ]  $R_{E,p}$ and $V_{E,p}$ are related by quotients which are direct sums of type
$$\cO_{\de_T}(-x_T-i,-i),\quad 0<i\leq |x_T|=-x_T\quad (x_T<0)$$
\item[(iv) ] $Q_{E,p}$ and $V_{E,p}$ are related by quotients which are direct sums of type
$$\cO_{\de_T}(-i,x_T-i),\quad 0<i\leq |x_T|=x_T\quad (x_T>0),$$
\ei
where we denote for simplicity $\de_T:=\de_{T\cup\{\infty\}}$ and $x_T:=x_{T,E,p}$. In particular, all pairs are related by quotients of type 
$$\cO(-a,*),\quad \cO(*,-a),\quad \text{with}\quad 0<a\leq\frac{S(E,p)}{2}.$$ 
\end{lemma}

\bp
This follows immediately from (\ref{resRS}), (\ref{L to R}),  (\ref{L to S}) and (\ref{RS to V}). The last statement follows by Lemma \ref{ineqX_lemma}. 
\ep

\begin{lemma}\label{ineqX_lemma}
Let $n=2s+2$, $E\subseteq N$, $e=|E|$, $p\in\ZZ$, $e+p$ even. Then for all $T$
\begin{equation}\label{score ineq}
|x_{T,E,p}|\leq \frac{S(E,p)}{2},
\end{equation}
where $S(E,p)$ is the score of the pair $(E,p)$ (Notation \ref{score notn}). 
Furthermore, 
$$ |x_{T,E,p}|=x_{T,E,p}=\frac{S(E,p)}{2}\quad\text{ if and only if }\quad T\subseteq E,\quad p\geq0$$

 \end{lemma}

\noindent The proof is straightforward and we omit it. Note, (\ref{alpha ineq}) is a particular case. 

\begin{cor}\label{next torsion}
Let $e=s+1$, $p\geq0$, and $(E,p)$ such that
$$S(E,p)=2\Bigl\lfloor\frac{s}{2}\Bigr\rfloor+2q,$$ 
with $p=2q-1$, $q\geq1$ if $s$ is even, and $p=2q-2$, $q\geq2$ if $s$ is odd. 
Assume the following objects are generated by $\cC$:
\bi
\item[(i) ] All torsion sheaves $\cO_{\de_T}(-a,0)$ for all $0<a<\Bigl\lfloor\frac{s}{2}\Bigr\rfloor+q$ and all $T$,
\item[(ii) ] The line bundles $R_{E,p}$, $Q_{E,p}$. 
\ei
Then $\cO_{\de_{T}}(-(\Bigl\lfloor\frac{s}{2}\Bigr\rfloor+q),0)$ with $T=E$ is generated by $\cC$. Here $\de_T:=\de_{T\cup\{\infty\}}$. 
\end{cor}
As $\cC$ is invariant under the action of $S_2$, it follows from Corollary \ref{next torsion} that  a similar statement holds when replacing 
$\cO_{\de_T}(-a,0)$ with $\cO_{\de_T}(0,-a)$.  

\bp
We claim that $V_{E,p}$ is generated by $\cC$. 
Since $R_{E,p}$ is generated by $\cC$ by assumption, using Lemma \ref{quotients}(iii), it suffices to prove that when $x_T<0$, 
$\cO(-x_T-i,-i)$ is generated by $\cA$, for all $0<i\leq |x_T|$, i.e., $|x_T|<\frac{s+1}{2}$. 
Since the assumptions on $q$ imply that $p>0$, we have  that
$$|x_T|=-x_T=\frac{e-p}{2}-|E\cap T|\leq \frac{e-p}{2}=\frac{s+1-p}{2}<\frac{s+1}{2},$$
and the claim follows. By Lemma \ref{quotients}(iv), the quotients relating 
$Q_{E,p}$ and $V_{E,p}$ have the form $\cO_{\de_T}(-i,x_T-i)$ for $0<i\leq x_T$. 
By Lemma  \ref{ineqX_lemma}, we have that 
$x_T\leq \frac{S(E,p)}{2}=\Bigl\lfloor\frac{s}{2}\Bigr\rfloor+q$, 
with equality if and only if $T\subseteq E$. Since $e=s+1$, we must have $T=E$.  It follows that all but one quotient, namely 
$\cO_{\de_{T}}(-(\Bigl\lfloor\frac{s}{2}\Bigr\rfloor+q),0)$ for $T=E$ (when $i=x_T=\frac{S(E,p)}{2}$) are already by assumption generated by $\cC$.
Note that this quotient appears exactly once. Since $Q_{E,p}$, $V_{E,p}$ are generated by $\cC$, it follows that this quotient is also. 
\ep

\begin{cor}\label{next torsion 2}
Let $q\in\ZZ$, $q>0$. Assume that 
$R_{E,p}$, $Q_{E,p}$ are generated by $\cC$ whenever 
$S(E,p)=2\Bigl\lfloor\frac{s}{2}\Bigr\rfloor+2q'$,  
with $0<q'\leq q$, and $e=s+1$. Then for all $T$, $\de_T:=\de_{T\cup\{\infty\}}$, 
the following  torsion sheaves are generated by $\cC$:
$$\cO_{\de_T}(-a,0),\quad \cO_{\de_T}(0,-a)\quad \text{when}\quad 0<a\leq\Bigl\lfloor\frac{s}{2}\Bigr\rfloor+q$$ 
\end{cor}

\bp
By the $S_2$ symmetry, it suffices to prove the statement  for $\cO_{\de_T}(-a,0)$. 
For any $q>0$, taking $E\subseteq N$ with $e=s+1$ and $p=2q-1$ when $s$ is even, or $p=2q-2$ when $s$ is odd, gives a  pair $(E,p)$ with 
$S(E,p)=2\Bigl\lfloor\frac{s}{2}\Bigr\rfloor+2q$. 
If $s$ is even, or if $s$ is odd and $q\geq2$, the assumptions of  Corollary \ref{next torsion} are satisfied. By induction on $q>0$, 
 $\cO_{\de_T}(-a,0)$ is generated by $\cC$ when $T=E$, $a=\Bigl\lfloor\frac{s}{2}\Bigr\rfloor+q$. 
 
The only case left is when $s$ is odd and $q=1$ ($p=0$). By assumption $R_{E,0}$, $Q_{E,0}$ are generated by $\cC$ if $e=s+1$ ($S(E,0)=s+1$). 
We have to prove that $\cO_{\de_T}(-\frac{s+1}{2},0)$ is generated by $\cC$. Taking $E=T$, $p=0$, we have that the pair
$(E,0)$ is in the range of Theorem \ref{even}. Hence, $L_{E,0}$ is in $\cC$. By Lemma \ref{quotients} and Lemma \ref{ineqX_lemma} 
$L_{E,0}$ and $R_{E,0}$ are related by quotients which are direct sums of sheaves in $\cA$, with only one quotient which is $\cO_{\de_T}(-\frac{s+1}{2},0)$
for $T=E$ (the only possibility to have $x_T=\frac{S(E,0)}{2}=\frac{s+1}{2}$ is when $T=E$). Note that this quotient appears exactly once. 
The statement follows.
\ep

\begin{lemma}\label{koszul}(Koszul resolutions)
Let $p\in\ZZ$, $E\subseteq N$. 
\bi
\item[(K1) ]
If $e\leq s+1$, letting $I\subseteq N\setminus E$, $|I|=s+1$, there is a long exact sequence:
$$0\ra Q_{E\cup I,p-s-1}\ra\ldots\ra\bigoplus_{J\subseteq I, |J|=j}Q_{{E\cup J},p-j}\ra\ldots\ra Q_{E, p}\ra0.$$
\item[(K2) ] If $e\geq s+1$, letting $I\subseteq E$, $|I|=s+1$, there is a long exact sequence:
$$0\ra R_{E,p}\ra\ldots\ra\bigoplus_{J\subseteq I, |J|=j}R_{{E\setminus J},p-j}\ra\ldots\ra R_{E\setminus I, p-s-1}\ra0.$$
\ei
\end{lemma}

\bp
We have $\bigcap_{i\in I}\de_{i\infty}=\emptyset$ and the boundary divisors $\{\de_{i\infty}\}_{i\in I}$ intersect transversely
(the divisors intersect properly and the intersection is smooth, being a Hassett space). 
It follows that there is a long exact sequence
$$0\ra \cO\left(-\sum_{i\in I} \de_{i\infty}\right)\ra\bigoplus_{j\in I}  \cO\left(-\sum_{i\in I\setminus\{j\}} \de_{i\infty}\right)
\ra\bigoplus_{j,k\in I}  \cO\left(-\sum_{i\in I\setminus\{j,k\}}\de_{i\infty}\right)\ra\ldots$$
$$\ldots\ra \bigoplus_{i\in I}  \cO(-\de_{i\infty})\ra\cO\ra 0.$$
Tensoring this long exact sequence by
$-\sum_{i\in E\setminus I}\de_{i\infty}-\frac{e-p}{2}\psi_{\infty},$ 
gives the second long exact sequence in the lemma. 
The first long exact sequence is obtained in a similar way 
by considering the Koszul resolution of the intersection of the boundary divisors $\{\de_{i0}\}_{i\in I}$.  
\ep

\begin{lemma}\label{slopes in K}
Assume $p\geq0$ and $E\subseteq N$ such that 
$$S(E,p)=2\Bigl\lfloor\frac{s}{2}\Bigr\rfloor+2q,$$
and the pair $(E,p)$ is such that $q\geq1$ if $s$ is even and $q\geq 2$ if $s$ is odd. 
In the notations of Lemma \ref{koszul}, we have: 
\bi
\item[(1) ] If $e\leq s+1$ then $Q_{E\cup J, p-j}$ in Lemma \ref{koszul}(K1) satisfies 
$S(E\cup J, p-j)\leq S(E,p)$. 
If equality holds, then $|p-j|<p$ if $j\neq0$
\item[(2) ] If $e\geq s+1$ then $R_{E\setminus J, p-j}$ in Lemma \ref{koszul}(K2) satisfies
$S(E\setminus J, p-j)\leq S(E,p)$. 
If equality holds, then $|p-j|<p$ if $j\neq0$. 
\ei
\end{lemma}

\bp 
We prove (1). We have  $S(E,p)=p+e$.  If $p-j\geq0$, then 
$$S(E\cup J, p-j)\leq (p-j)+e+j=p+e=S(E,p),$$
and clearly $|p-j|=p-j<p$ if $j\neq0$. 
If $p-j<0$, we prove that the inequality on slopes is strict. We have
$$S(E\cup J, p-j)\leq (j-p)+(n-e-j)=n-p-e<e+p=S(E,p),$$
since $S(E,p)=e+p=2\Bigl\lfloor\frac{s}{2}\Bigr\rfloor+2q>s+1$. 

We prove (2). We have  $S(E,p)=p+(n-e)$.  If $p-j\geq0$, then 
$$S(E\setminus J, p-j)\leq (p-j)+(n-e+j)=p+(n-e)=S(E,p),$$
and clearly $|p-j|=p-j<p$ if $j\neq0$. 
If $p-j<0$, we prove that the inequality on slopes is strict. We have
$$S(E\setminus J, p-j)\leq (j-p)+(e-j)=e-p<p+n-e=S(E,p),$$
since $e-p<s+1$, as
$S(E,p)=p+n-e=2\Bigl\lfloor\frac{s}{2}\Bigr\rfloor+2q>s+1$.
\ep

\bp[Proof of Theorem \ref{full even}]

\noindent {\bf Case $s$ even.}
For any $(E,p)$ write the score $S(E,p)$ as
\begin{equation}
S(E,p)=s+2q.
\end{equation}
Note that if $q\leq0$ then $L_{E,p}$ is already in $\cC$ (Remark \ref{reformulate}). Moreover, if $q\leq 0$, by Lemma \ref{quotients}  $R_{E,p}$ and $Q_{E,p}$ are related by quotients which are direct sums of torsion sheaves of the form $\cO(-a,*)$ or $\cO(*,-a)$, with $0<a\leq |x_T|$. As 
$|x_T|\leq \frac{S(E,p)}{2}\leq\frac{s}{2}<\frac{s+1}{2}$, such quotients are in $\cA$. 
 
We prove by induction on $q\geq0$, and for equal $q$, by induction on $|p|$, that 
$R_{E,p}$, $Q_{E,p}$ with $S(E,p)=s+2q$  
are generated by $\cC$. By Corollary \ref{next torsion 2}, it follows that 
all $\cO_{\de_T}(-a,-b)$ are generated by $\cC$. Then Lemma \ref{quotients} implies then that all line bundles $L_{E,p}$ are generated by $\cC$. 

We now prove the inductive statement. 
For $q\leq0$, we already proved that $R_{E,p}$, $Q_{E,p}$ are generated by $\cC$. 
Assume $q\geq1$. Take a pair $(E,p)$ with score $S(E,p)=s+2q$. Using the $S_2$ symmetry, we may assume $p\geq0$. 
For any $(E',p')$ with strictly smaller score than $s+2q$, or equal score and strictly smaller $|p|$, we have by induction that $Q_{E',p'}$, $R_{E',p'}$ are
generated by $\cC$.  

If $e\leq s+1$, we apply Lemma \ref{koszul} and get a resolution for $Q_{E,p}$. Using Lemma \ref{slopes in K}(i), 
all terms in the resolution are generated by $\cC$ by induction. Hence, $Q_{E,p}$ is generated by $\cC$ if $e\leq s+1$. 
Similarly, using Lemma \ref{koszul}, Lemma \ref{slopes in K}(ii) and induction, $R_{E,p}$ is generated by $\cC$ if $e\geq s+1$. 

We have that both $Q_{E,p}$, $R_{E,p}$ are generated by $\cC$ if $e=s+1$. 
By Corollary \ref{next torsion 2} and the induction assumption, 
$\cO_{\de_T}(-a,0)$, $\cO_{\de_T}(0,-a)$ if $0<a\leq\frac{s}{2}+q$ are generated by $\cC$. 
By Lemma \ref{quotients} we have that $L_{E,p}$ is related to each of $Q_{E,p}$, $R_{E,p}$ by quotients which are direct sums of $\cO_{\de_T}(-a,*)$, $\cO_{\de_T}(*,-a)$ with $0<a\leq\frac{S(E,p)}{2}=\frac{s}{2}+q$. Since for any $e\neq s+1$, one of
$Q_{E,p}$, $R_{E,p}$ is generated by $\cC$, it follows  that $L_{E,p}$ is generated by $\cC$. 

\smallskip

\noindent {\bf Case $s$ odd.}
For any $(E,p)$ write the score $S(E,p)$ as
\begin{equation}
S(E,p)=(s-1)+2q.
\end{equation}

We prove by induction on $q\geq0$, and for equal $q$, by induction on $|p|$, that the line bundles 
$R_{E,p}$ and $Q_{E,p}$ with $S(E,p)=(s-1)+2q$  
are generated by $\cC$. 
This proves the theorem, as Corollary \ref{next torsion 2}  gives that all torsion sheaves supported on boundary are generated by $\cC$. 
The inductive argument we did for $s$ even goes through verbatim if $q\geq2$ (the assumption is used in Lemma \ref{slopes in K}).
Hence, we only need to prove that $R_{E,p}$, $Q_{E,p}$ are generated by $\cC$ for $q=0$ and $q=1$. We may assume w.l.o.g. that $p\geq0$. 

Assume $q=0$. Fix a pair $(E,p)$ with $S(E,p)=s-1$. 
Then $(E,p)$ is in the range of Theorem \ref{even} and $L_{E,p}$ is in $\cC$. As in the previous case, by Lemma \ref{quotients}, the line bundles 
$R_{E,p}$, $Q_{E,p}$ are related to $L_{E,p}$ by quotients generated by $\cA$. Hence, $R_{E,p}$, $Q_{E,p}$ are generated by $\cC$.  

Assume now $q=1$ and fix a pair $(E,p)$ with $S(E,p)=s+1$.
\begin{claim}\label{first bdry}
 $\cO_{\de_T}(-\frac{s+1}{2},0)$, $\cO_{\de_T}(0,-\frac{s+1}{2})$ are generated by $\cC$. 
 \end{claim}
 
\bp
By Corollary \ref{next torsion 2}, it suffices to prove that $R_{E,0}$, $Q_{E,0}$ are generated by $\cC$ for some $E$ with $e=|E|=s+1$. Take such an $E$. 
By Remark \ref{symmetry}, $R_{E,0}$ and $Q_{E,0}$ are exchanged by the action of $S_2$. Hence, by symmetry, it suffices to prove that $R_{E,0}$ 
is generated by $\cC$. Consider the resolution in Lemma \ref{koszul}(ii) for $(E,0)$, with $I=E$. The terms that appear, other than $R_{E,0}$,  
are $R_{E\setminus J, -j}$, with $J\subseteq E$, $j>0$. For all $j\geq0$, $S(E\setminus J, -j)=s+1$ and all  $(E\setminus J, -j)$ 
are in the range of Theorem \ref{even}. Hence, $L_{E\setminus J, -j}$ are generated by $\cC$. 

We claim that if $j>0$, the quotients relating $R_{E\setminus J, -j}$ to $L_{E\setminus J, -j}$ are generated by $\cA$. 
By Lemma \ref{quotients} the quotients relating $R_{E\setminus J, -j}$ to $L_{E\setminus J, -j}$ are 
$$\cO_{\de_T}(-x_T+i,i),\quad 0\leq i<x_T=x_{T,E\setminus J,-j}\quad \text{where}$$
$$x_T=|(E\setminus J)\cap T|-\frac{s+1}{2}\leq |E\setminus J|-\frac{s+1}{2}\leq s-\frac{s+1}{2}=\frac{s-1}{2}.$$
The claim follows. It follows that for $j>0$, $R_{E\setminus J, -j}$ is generated by $\cC$. Using the resolution, it follows that $R_{E,0}$ is generated by $\cC$. 
\ep

Assume that $e\leq s+1$. Then $(E,p)$ is in the range of Theorem \ref{even} and $L_{E,p}$ is in $\cC$. Since $R_{E,p}$, $Q_{E,p}$ are related to 
$L_{E,p}$ by quotients $\cO_{\de_T}(-a,*)$, $\cO_{\de_T}(*,-a)$ with $0<a\leq\frac{S(E,p)}{2}=\frac{s+1}{2}$, it follows by Claim \ref{first bdry} that $R_{E,p}$, $Q_{E,p}$ are 
generated by $\cC$. 

Assume now that $e>s+1$. Then $(E,p)$ is not in the range of Theorem \ref{even}. Note that it suffices to prove that $R_{E,p}$ is generated by $\cC$, since by Lemma \ref{quotients} $R_{E,p}$, $L_{E,p}$ are related by quotients which are direct sums of $\cO_{\de_T}(-a,*)$ with $0<a\leq\frac{S(E,p)}{2}=\frac{s+1}{2}$
(generated by $\cC$ by Claim \ref{first bdry}). To prove $R_{E,p}$ is generated by $\cC$, we do an 
induction on $e\geq s+1$ (for $(E,p)$ of fixed score $s+1$) by using a resolution as in Lemma \ref{koszul} for $R_{E,p}$. 
\ep

\begin{rmk}\label{elaborate2} 
For $n=2s+2\geq 2$, the exceptional collection on $Z_n$ given in \cite[Theorem 1.15]{CT_partII} consists of: 

(i) The same torsion sheaves $\cO_{\de_T}(-a,-b)$ as in Theorem \ref{even}.

(ii) The line bundles in the so-called group $1$ (group $1A$ and group $1B$ of that theorem coincide in this case): for all $E\subseteq N$, with $e=|E|$ even, 
\begin{equation}\label{F0E n even}
F_{0,E}=\frac{e}{2}\psi_\infty+\sum_{j\in E} \de_{j\infty}-\sum_{\frac{e}{2}-|E\cap T|>0} \big(\frac{e}{2}-|E\cap T|\big)\de_{T\cup\{\infty\}}. 
\end{equation}

The line bundles $F_{0,E}$ are defined in \cite{CT_partII} as $R\pi_*(N_{0,E})$, for certain line bundles $N_{0,E}$ on the universal family over $Z_n$. One checks directly (or see the proof of \cite[Lemma 5.8]{CT_partII}) 
that $N_{0,E}$ restrict trivially to every component of any fiber of the universal family 
$\pi:\cU\ra Z_n$. Hence, 
$$N_{0,E}=\pi^*F_{0,E},\quad F_{0,E}=\si_u^*N_{0,E},$$ 
for any marking $u$. In particular, for $u\in\{0,\infty\}$, we obtain formula 
(\ref{F0E n even}). 

(iii) The objects in the so-called group $2B$, which in this case are line bundles (corresponding only to the $J=\emptyset$ term in 
\cite[Notation 11.5]{CT_partII}):
$$\cTT_{l,\{u\}\cup E}:=\frac{e-l-1}{2}\psi_u+\sum_{j\in E}\de_{ju}-\sum_{\frac{e-l-1}{2}-|E\cap T|>0}\left(\frac{e-l-1}{2}-|E\cap T|\right)\de_{T\cup\{u\}}$$
where  $u\in\{0,\infty\}$, $E\subseteq N$, $e=|E|$, $l\in\ZZ$, $l\geq0$ such that $|E\cap\{u\}|+l$ is even (i.e., $e+l$ is odd), subject to the condition
$$l+\min\{e,n+1-e\}\leq s\quad (\text{group}\quad 2B).$$

The formula generalizing both expressions  in (ii) and (iii) is 
$$\frac{e-p}{2}\psi_u+\sum_{j\in E}\de_{ju}-\sum_{\frac{e-p}{2}-|E\cap T|>0}\left(\frac{e-p}{2}-|E\cap T|\right)\de_{T\cup\{u\}}$$
which, when $u=\infty$, is exactly the line bundle $V^\vee_{E,p}$ (the dual of $V_{E,p}$ - see (\ref{extraLB}). 
Hence, the group $2B$  with $l=p-1$, $u=\infty$ recovers all the $\{V^\vee_{E,p}\}$ 
when $p>0$. Similarly, the group $2B$ with $l=-p-1$, $u=0$ recovers all the $\{V^\vee_{E,p}\}$ when $p<0$. The elements of group $1$ recover all the 
$\{V^\vee_{E,p}\}$ when $p=0$. 
A similar proof as in this section will prove that the collection in  \cite[Theorem 1.15]{CT_partII} - the torsion sheaves (i) and 
the line bundles $\{V^\vee_{E,p}\}$, for  $(E,p)$ as in Theorem \ref{even}-  is a full exceptional collection. 
\end{rmk}


\section{Pushforward of the exceptional collection on the Losev--Manin space $\LM_N$  to $Z_N$}\label{LM}

We refer to \cite{CT_partI} for background on Losev--Manin spaces. 
Recall that the Losev--Manin moduli space  $\LM_N$ is the Hassett space with markings $N\cup\{0,\infty\}$
and weights $(1,1,\frac{1}{n},\ldots,\frac{1}{n})$, where  $n=|N|$. 
The space $\LM_N$ parametrizes  nodal linear chains of $\bP^1$'s marked by $N\cup\{0,\infty\}$ 
with $0$ is on the left tail and $\infty$ is on the right tail of the chain. 
Both   $\psi_0$ and $\psi_\infty$ 
induce birational morphisms $\LM_N\to\bP^{n-1}$ (Kapranov models) which realize $\LM_N$ as an iterated blow-up of $\bP^{n-1}$ in $n$ points (standard basis vectors) followed by blowing up $n\choose 2$ proper transforms of lines connecting points, etc.
In particular, $\LM_N$ is a toric variety of dimension $n-1$. Its  toric orbits (or their closures, the boundary strata
as a moduli space) are given by partitions $N=N_1\sqcup\ldots\sqcup N_k$, $|N_i|>0$ for all $i$, which
correspond to boundary strata
$$Z_{N_1,\ldots, N_k}=\de_{N_1\cup\{0\}}\cap \de_{N_1\cup N_2\cup\{0\}}\cap\ldots\cap
\de_{N_1\cup \ldots\cup N_{k-1}\cup\{0\}}$$ 
which parametrizes (degenerations of) linear chains of $\bP^1$'s with points marked by, respectively,
$N_1\cup\{0\}$, $N_2$,\ldots, $N_{k-1}$, $N_k\cup\{\infty\}$. 
We can identify
$$Z_{N_1,\ldots, N_k}\simeq \LM_{N_1}\times\ldots\times\LM_{N_k},$$ 
where the left node of every $\PP^1$ is marked by $0$
and the right node by $\infty$. 

There are forgetful maps
$\pi_K:\, \LM_N\to \LM_{N\setminus K}$,  
for all $K\subseteq N$, $1\le|K|\le n-1$, given by forgetting points marked by $K$ and stabilizing.

\begin{defn}[\protect{\emph{cf.} \cite[Definition 1.4]{CT_partI}}]
The cuspidal block $D^b_{cusp}(\LM_N)$ consists of objects $E\in D^b(\LM_N)$ such that for all $i\in N$ we have 
$$R{\pi_i}_*E=0.$$
\end{defn}

\begin{prop}[\protect{\emph{cf.} \cite[Proposition 1.8]{CT_partI}}]\label{LM decomp}
There is  a semi-orthogonal decomposition
$$D^b(\LM_N)=\langle D^b_{cusp}(\LM_N), \ \{\pi_K^*D^b_{cusp}(\LM_{N\setminus K})\}_{K\subset N},\ \cO\rangle$$
where subsets $K$ with $1\le |K|\le n-2$ are ordered by  increasing cardinality.
\end{prop}

\begin{defn}[\protect{\emph{cf.} \cite[Definition 1.9]{CT_partI}}]
Let $\bG_N=\{G_1^\vee,\ldots,G_{n-1}^\vee\}$ be the set of following line bundles on $\LM_N$:
$$G_a=a\psi_{0}-(a-1)\sum_{k\in N}\de_{k0}-(a-2)\sum_{k,l\in N}\de_{kl0}-\ldots-\sum_{J\subset N, |J|=a-1}\de_{J\cup\{0\}}
$$
for every $a=1,\ldots,n-1$.
Let $\hat\bG$ be the collection of sheaves of the form
$$\cT=(i_Z)_*\cL,\quad \cL=G_{a_1}^\vee\boxtimes\ldots\boxtimes G_{a_t}^\vee$$
for all \emph{massive} strata $Z=Z_{N_1,\ldots,N_t}$, i.e., such that $N_i\ge 2$ for every $i$ and for all $1\leq a_i\leq |N_i|-1$. 
Here $i_Z:\,Z\hookrightarrow \LM_N$ is the inclusion map.
If $t=1$ we get line bundles $\bG_N$ and for $t\geq2$ these sheaves are torsion sheaves.
 \end{defn}
 
\begin{thm}[\protect{\emph{cf.} \cite[Theorem 1.10]{CT_partI}}]\label{LM main}
$\hat\bG$ is a full exceptional  collection in $D^b_{cusp}(\LM_N)$,
which is invariant under the group $S_2\times S_N$.
\end{thm}

Clearly, by Theorem \ref{LM main}, Proposition \ref{LM decomp} and adjointness, we have the following
\begin{cor}\label{S is enough}
If $E\in D^b(Z_N)$ is such that $R\Hom(E, F)=0$ for all $F$ of the form $Rp_*({\pi_K}^*\hat\GG)$, for all $K\subseteq N$, including $K=\emptyset$, then $E=0$. 
\end{cor}

We now proceed to calculate the objects in the collection $Rp_*({\pi_K}^*\hat\GG)$. 
\begin{prop}\label{pull-backs push}
Let $p:\LM_N\ra Z_N$ be the reduction map. 

\bi
\item[(1) ] For all $I\subseteq N$ with $0\leq |I|\leq n-2$ and all $1\leq a\leq n-|I|-1$, we have
$$Rp_*\big(\pi_I^*G_a^\vee\big)=-a\psi_0-\sum_{j\in N\setminus I}\de_{j0}\quad \text{if $n$ is odd},$$
$$Rp_*\big(\pi_I^*G_a^\vee\big)=-a\psi_0-\sum_{j\in N\setminus I}\de_{j0}+\sum_{J\subseteq N, |J|=\frac{n}{2}, |J\cap (N\setminus I)|< a} \big(a- |J\cap (N\setminus I)|\big) \de_{J\cup\{0\}},$$
if $n$ is even. Moreover, $Rp_*\cO=\cO$. 
\item[(2) ] If $n$ is odd, all the torsion sheaves and their pull-backs, i.e, all sheaves $\cT$ in the collection  $\hat \GG$ not considered in (1), have $Rp_*(\cT)=0$. 
\item[(3) ] If $n$ is even, we have 
$Rp_*\big(G^\vee_{a_1}\boxtimes\ldots\boxtimes G^\vee_{a_t}\big)=0$,  
except for sheaves $G^\vee_a\boxtimes G^\vee_b$ with support $Z=\LM_{N_1}\times \LM_{N_2}$, where $|N_1|=|N_2|=\frac{n}{2}$, when
$$Rp_*\big(G^\vee_a\boxtimes G^\vee_b\big)=\cO(-a)\boxtimes\cO(-b),$$
where we use the identification $p(Z)=\PP^{\frac{n}{2}-1}\times \PP^{\frac{n}{2}-1}$. 
\item[(4) ] If  $n$ is even, $I\neq\emptyset$ and $\cT\in \hat \GG_{N\setminus I}$ is a torsion sheaf, then either
$$Rp_*\big(\pi_I^*\cT\big)=0,$$ or $Rp_*\big(\pi_I^*\cT\big)$ is generated by the sheaves 
$\cO(-a)\boxtimes\cO(-b)$  supported on the images $\PP^{\frac{n}{2}-1}\times \PP^{\frac{n}{2}-1}$ of strata $\LM_{N_1}\times \LM_{N_2}$ with $|N_1|=|N_2|=\frac{n}{2}$ and with $0<a,b\leq \frac{n}{2}-1$. 
\ei
\end{prop}

We use here that if $n=2s+2$ is even, 
the restriction of the map $p$ to a stratum of the form $\LM_{s+1}\times \LM_{s+1}$ is a product of reduction maps of type 
$\LM_{s+1}\ra\M_{\ba}$, where $\ba=(1,\frac{1}{2}+\eta,\frac{1}{n},\ldots, \frac{1}{n})$ (with $\frac{1}{n}$ appearing $(s+1)$ times). 
By \cite[Remark 4.6]{Ha}, we have $\M_{0,\cA}=\PP^s$ (the Kapranov model of $\LM_{s+1}$ with respect to the first marking).

\bp[Proof of Proposition \ref{pull-backs push}]
Throughout, we denote $s:=\left\lfloor\frac{n-1}{2}\right\rfloor$. 
We first prove (1). As $p$ is a birational morphism between smooth projective varieties, we have $Rp_*\cO=\cO$. 
We write $\pi_I^*G_a^\vee$ and $p^*\big(-a\psi_0-\sum_{j\in N\setminus I} \de_{j0}\big)$
in the Kapranov model with respect to the $0$ marking. We denote
$$H:=\psi_0,\quad E_J:=\de_{J\cup\{0\}}\quad (J\subseteq N,\quad |J|\leq n-2),$$
be the hyperplane class and the exceptional divisors respectively. We have: 

$$\pi_I^*G_a^\vee=-aH+\sum_{J\subseteq N, |J|\geq1, |J\cap (N\setminus I)|< a} \big(a- |J\cap (N\setminus I)|\big) E_J.$$ 
$$p^*\big(-a\psi_0-\sum_{j\in N\setminus I} \de_{j0}\big)=-aH+\sum_{J\subseteq N, 1\leq|J|\leq s} \big(a- |J\cap (N\setminus I)|\big) E_J,$$
where the last equality follows by Corollary \ref{pull by p}. It follows that 
$$\pi_I^*G_a^\vee=p^*\big(-a\psi_0-\sum_{j\in N\setminus I} \de_{j0}\big)+\Sigma^1+\Sigma^2,$$
where $\Sigma^1$ consists of all the terms that appear in $\pi_I^*G_a^\vee$, but do not appear in 
$p^*\big(-a\psi_0-\sum_{j\in N\setminus I} \de_{j0}\big)$, and
$\Sigma^2$ consists of the terms that appear in $p^*\big(-a\psi_0-\sum_{j\in N\setminus I} \de_{j0}\big)$, but do not in $\pi_I^*G_a^\vee$, 
taken with a negative sign:
$$\Sigma^1=\sum_{J\subseteq N, |J|\geq1, |J\cap (N\setminus I)|< a, |J|>s} \big(a- |J\cap (N\setminus I)|\big) E_J,$$
$$\Sigma^2=\sum_{J\subseteq N, |J|\geq1, |J\cap (N\setminus I)|> a, |J|\leq s} \big(|J\cap (N\setminus I)|-a\big) E_J.$$

When $|J|\leq s$, the codimension of $p(E_J)$ in $Z_N$ is $|J|$. For the terms in the sum $\Sigma^2$, the coefficient of $E_J$ 
satisfies $$|J\cap (N\setminus I)|-a\leq |J|-1=\codim(p(E_J))-1.$$
Hence, one may apply Lemma \ref{push exceptional} successively to the terms of the sum $\Sigma^2$. We use here that the map $p$ can be decomposed into a sequence of blow-ups, with exceptional divisors $\de_{J\cup\{0\}}$, $\de_{J\cup\{\infty\}}$, with $1\leq |J|\leq s$, in order of increasing $|J|$. Note that 
the divisors $E_J$ with fixed $|J|$ are disjoint. 

Similarly, when $|N\setminus J|\leq s$,  the codimension of $p(E_J)$ in $Z_N$ is $|N\setminus J|$.  For the terms in the sum $\Sigma^1$, 
the coefficient of $E_J$ with  $|N\setminus J|\leq s$, satisfies 
$$a-|J\cap (N\setminus I)|\leq n-1-|I|-|J\cap (N\setminus I)|\leq n-1-|J|=\codim(p(E_J))-1,$$
so one may apply again Lemma \ref{push exceptional} to the terms of the sum $\Sigma^1$ which satisfy $|N\setminus J|\leq s$. 
When $n=2s+1$, the inequality   $|N\setminus J|\leq s$ is equivalent to $|J|>s$. However, when $n=2s+2$, the inequality  
$|N\setminus J|\leq s$ is equivalent to $|J|>s+1$. Hence, in the case when $n=2s+2$, one is left with the terms in the sum $\Sigma^1$ 
that have $|J|=s+1$. This proves (1). 

\smallskip

Now we turn to the torsion objects, i.e., objects of the form $\pi_I^*(\cT)$, where 
$$\cT={i_Z}_*\big(G^\vee_{a_1}\boxtimes\ldots\boxtimes G^\vee_{a_t}\big),\quad 
Z=\LM_{K_1}\times\ldots\times\LM_{K_t},$$ where $I\subseteq N$, $N\setminus I=K_1\sqcup\ldots\sqcup K_t$ and $|K_j|\geq2$, for all $j$. 
Consider first the case when $I=\emptyset$. If $|K_1|\leq s$, the map $Z\ra p(Z)$ is a product of reduction maps, the first of which is the constant map
 $\LM_{K_1}\ra pt$. It follows in this case that $Rp_*(\cT)=0$, since $R\Gamma(G^\vee_{a_i})=0$. The same argument applies when $|K_t|\leq s$. 
It follows that $Rp_*(\cT)=0$, except possibly in the case when $n=2s+2$, $t=2$ and $|K_1|=|K_2|=s+1$. In this case, the map 
$Z\ra p(Z)$ is a product of Kapranov maps $\LM_{s+1}\times \LM_{s+1}\ra \PP^s\times\PP^s$, and it follows 
(for example by Lemma \cite[Lemma 5.7]{CT_partI}) that in this case
$Rp_*(\cT)=\cO_{\PP^s}(-a)\boxtimes\cO_{\PP^s}(-b)$.  
This proves (3) and the case $I=\emptyset$ of (2). 

Consider now the case when $I\neq\emptyset$. To compute $Rp_*(\pi_I^*\cT)$, consider
the boundary divisors $D_1, D_2, \ldots, D_{t-1}$ whose intersection is $Z$, denoting 
$$D_i=\LM_{K_1\sqcup\ldots\sqcup K_i}\times \LM_{K_{i+1\sqcup\ldots\sqcup K_t}}\quad (i=1,\ldots, t-1).$$
For the remaining part of the proof, we denote for simplicity 
$$K'=K_1,\quad K''=K_2\sqcup\ldots\sqcup K_t$$
and consider the canonical inclusions
$$i_1: Z\hra D_1=\LM_{K'}\times\LM_{K''},\quad i_{D_1}: D_1\ra \LM_{N\setminus I}.$$ 
We resolve ${i_1}_*\cO_Z$ using the Koszul complex 
$$\ldots \ra\oplus_{2\leq i<j\leq t}\cO(-D_i-D_j)_{|D_1}\ra  \oplus_{2\leq i\leq t}\cO(-D_i)_{|D_1}\ra\cO_{D_1}\ra {i_1}_*\cO_Z\ra0.$$
By our choice of $D_1$, for all $2\le i \le t$ we have $\cO(D_i)_{|D_1}=\cO\boxtimes\cO(D'_i)$, for the corresponding boundary divisor 
on $\LM_{K''}$:
$$D'_i=\LM_{K_2\sqcup\ldots\sqcup K_i}\times \LM_{K_{i+1\sqcup\ldots\sqcup K_t}}.$$

By Lemma  \ref{lift}, we may choose a line bundle $\cM$ on $\LM_{K''}$ such that the restriction of $\cM$ to the massive stratum 
$\LM_{K_2}\times\ldots\times\LM_{K_t}$ is $G^\vee_{a_2}\boxtimes\ldots\boxtimes G^\vee_{a_t}$ and 
$\cM\otimes\cO(-D'_{i_1}-\ldots-D'_{i_k})$ is acyclic for any $2\le i_1<\ldots<i_k\le t$. 

Consider the line bundle $\cL=G^\vee_{a_1}\boxtimes\cM$ on $D_1$.  Then $\cL_{|Z}=G^\vee_{a_1}\boxtimes\ldots\boxtimes G^\vee_{a_t}$.
We now: (1) Tensor the Koszul sequence with $\cL$, (2) Apply $R{i_{D_1}}_*(-)$, and (3) Apply $L\pi_I^*(-)$. Since $\pi_I$ is flat, we obtain a resolution 
for $\pi_I^*\cT$ with sheaves whose support is contained in $\pi_I^{-1}(D_1)$. To prove (4) and the remaining part of (2), it suffices to show that 
for all $2\le i_1<\ldots<i_k\le t$ 
$$Rp_*\pi_I^*R{i_{D_1}}_*\big(\cL\otimes \cO(-D_{i_1}-\ldots-D_{i_k})_{|D_1}\big)$$ 
is $0$ when $n$ is odd, or generated by the sheaves $\cO(-a)\boxtimes\cO(-b)$ ($a,b>0$) supported on the divisors 
$\PP^{\frac{n}{2}-1}\times \PP^{\frac{n}{2}-1}$ as in (4), when $n$ is even. 
Here we need the same statement also for $Rp_*\pi_I^*R{i_{D_1}}_*\big(\cL\big)$ (i.e., $k=0$). 
Note that 
$$\cL\otimes \cO(-D_{i_1}-\ldots-D_{i_k})_{|D_1}=G^\vee_{a_1}\boxtimes\big(\cM\otimes\cO(-D'_{i_1}-\ldots-D'_{i_k})\big).$$ 

There is a commutative diagram
\begin{equation*}
\begin{CD}
\pi_I^{-1}(D_1)     @>{i_{\pi^{-1}(D_1)}}>>  \LM_N@>p>> Z_N\\
@V{\rho_I}VV        @V{\pi_I}VV\\
D_1  @>{i_{D_1}}>>  \LM_{N\setminus I} 
\end{CD}
\end{equation*}
where ${i_{\pi_I^{-1}(D_1)}}$ is the canonical inclusion map and $\rho_I$ is the restriction of ${\pi_I}$ to $\pi_I^{-1}(D_1)$. 
Let $q=p \circ i_{\pi^{-1}(D_1)}$. 
As $\pi_I$ is flat, we have
$$Rp_*\pi_I^*R{i_{D_1}}_*\big(\cL\otimes \cO(-D_{i_1}-\ldots-D_{i_k})_{|D_1}\big)=
Rq_*\rho_I^*\big(\cL\otimes \cO(-D_{i_1}-\ldots-D_{i_k})_{|D_1}\big).$$

The preimage  $\pi_I^{-1}(D_1)$ has several components $B_{I_1,I_2}$:
$$B_{I_1,I_2}=\LM_{K'\cup I_1}\times \LM_{K''\cup I_2}\quad \text{for every partition}\quad I=I_1\sqcup I_2$$ 

We order the set $\{B_{I_1,I_2}\}$ as follows: $B_{I_1,I_2}$ must come before $B_{J_1,J_2}$ if $|I_1|>|J_1|$ and in a random order if $|I_1|=|J_1|$. Hence, if $B_{I_1,I_2}$ comes before $B_{J_1,J_2}$, then $B_{I_1,I_2}\cap B_{J_1,J_2}\neq\emptyset$ if and only if $J_1\subsetneq I_1$, in which case, the intersection takes the form
 $$B_{I_1,I_2}\cap B_{J_1,J_2}=\LM_{K'\cup J_1}\times \LM_{(I_1\setminus J_1)}\times \LM_{K''\cup I_2}.$$
For simplicity, we rename the resulting ordered sequence as $B_1, B_2, \ldots, B_r$. A consequence of the ordering is that 
$B_r$ is the component $B_{I_1,I_2}$ with $I_1=\emptyset$, $I_2=I$, and if $1\leq i\leq r-1$ and 
$B_i$ is $B_{I_1,I_2}=\LM_{K'\cup I_1}\times \LM_{K''\cup I_2}$, then
$$\cO_{B_{i}}(B_{i+1}+\ldots+B_r)=\cO(\sum_{J\subsetneq I_1}\de_{J\cup K'\cup\{0\}})\boxtimes\cO=
\cO(\sum_{\emptyset\neq S\subseteq I_1}\de_{S\cup\{x\}})\boxtimes\cO,$$ 
where the first sum runs over all $J\subsetneq I_1$ (including $J=\emptyset$), while for the second sum we use the identification 
$$\de_{J\cup K'\cup\{0\}}=\de_{(I_1\setminus J)\cup\{x\}}=\LM_{J\cup K'}\times\LM_{I_1\setminus J},$$ 
as divisors in $\LM_{K'\cup I_1}$ (with $x$ being the attaching point). 
Consider now the following exact sequences resolving $\cO_{\pi_I^{-1}(D_1)}=\cO_{B_1\cup\ldots\cup B_r}$:

$$0\ra \cO_{B_1\cup\ldots\cup B_{r-1}}(-B_r)\ra\cO_{B_1\cup\ldots\cup B_r}\ra\cO_{B_r}\ra 0,$$
$$0\ra \cO_{B_1\cup\ldots\cup B_{r-2}}(-B_{r-1}-B_r)\ra\cO_{B_1\cup\ldots\cup B_{r-1}}(-B_r)\ra\cO_{B_{r-1}}(-B_r)\ra 0,$$
$$\vdots$$
$$0\ra \cO_{B_1}(-B_2-\ldots-B_r)\ra\cO_{B_1\cup B_2}(-B_3-\ldots-B_r)\ra\cO_{B_2}(-B_3-\ldots-B_r)\ra 0.$$

We tensor all the above exact sequences with $\cL\otimes \cO(-D_{i_1}-\ldots-D_{i_k})_{|D_1}$ and apply first 
$\rho_I^*(-)$, then $Rq_*(-)$.  As the restriction of the map ${\rho_I}$ to a component $B_i$ of the form $B_{I_1,I_2}$ for some partition $I=I_1\sqcup I_2$,
is the product of forgetful maps 
$\pi_{I_1}\times\pi_{I_2}$, it follows that, if $i\neq r$, then
\begin{align*}
\cO_{B_{i}}(-B_{i+1}-\ldots-B_r)&\otimes\rho_I^*\big(\cL\otimes \cO(-D_{i_1}-\ldots-D_{i_k})_{|D_1}\big)=\\
&\big(\pi_{I_1}^*G^\vee_{a_1}\otimes\cO(-\sum_{\emptyset\neq S\subseteq I_1}\de_{S\cup\{x\}})\big)\boxtimes 
\pi_{I_2}^*\big(\cM\otimes\cO(-D'_{i_1}-\ldots-D'_{i_k})\big),
\end{align*}
while 
$$\cO_{B_{s}}\otimes\rho_I^*\big(\cL\otimes \cO(-D_{i_1}-\ldots-D_{i_k})_{|D_1}\big)=G^\vee_{a_1}\boxtimes \pi_I^*\big(\cM\otimes\cO(-D'_{i_1}-\ldots-D'_{i_k})\big).$$
(Recall that $B_r$ corresponds to the partition $I_1=\emptyset$, $I_2=I$.)

We claim that both components of all the above sheaves are acyclic. To prove the claim, 
recall that $\cM\otimes\cO(-D'_{i_1}-\ldots-D'_{i_k})$ is acyclic by the choice of $\cM$. We are left to prove that $\pi_{I_1}^*(G^\vee_{a_1})\otimes\cO(-\sum_{\emptyset\neq S\subseteq I_1}\de_{S\cup\{x\}})$ is acyclic when $I_1\neq\emptyset$. Since we may rewrite the line bundle $G^\vee_{a_1}$ using the $x$ marking, we are done by the following:
\begin{claim}
Consider the forgetful map $\pi_I: \LM_{N\cup I}\ra \LM_N$ for some subset $I\neq\emptyset$. 
For all $1\leq b\leq |N|-1$, the line bundle 
$\pi_{I}^*(G^\vee_{b})\otimes\cO(-\sum_{\emptyset\neq S\subseteq I}\de_{S\cup\{0\}})$ is acyclic. 
\end{claim}
\bp
Using the Kapranov model with respect to the $0$ marking, we have
$$\pi_{I}^*(G^\vee_{b})=-bH+\sum_{J\subseteq N\cup I, |J\cap N|< b}(b-|J\cap N|) E_J,\quad
\cO\left(-\sum_{\emptyset\neq S\subseteq I}\de_{S\cup\{0\}}\right)=-\sum_{\emptyset\neq S\subseteq I}E_S.$$
As $b-|J\cap N|-1\geq0$, the result follows by Lemma \ref{easy acyclic}.
\ep

Recall that the map $p$ either contracts 
$B_{I_1,I_2}=\LM_{K'\cup I_1}\times \LM_{K''\cup I_2}$
by mapping $\LM_{K'\cup I_1}$ to a point if $|I_1+K'|<\frac{n}{2}$, or by mapping 
$LM_{K''\cup I_2}$ to a point if $|I_2+K''|<\frac{n}{2}$), 
or, we have $|I_1+K'|=|I_2+K''|=\frac{n}{2}$ and $p(B_{I_1,I_2})$ is a divisor in $Z_N$ which is isomorphic to 
$\PP^{\frac{n}{2}-1}\times \PP^{\frac{n}{2}-1}$. Hence, 
$$Rq_*\big(\cO_{B_{i}}(-B_{i+1}-\ldots-B_s)\otimes\rho_I^*\big(\cL\otimes \cO(-D_{i_1}-\ldots-D_{i_k})_{|D_1}\big)\big),$$
$$Rq_*\big(\cO_{B_{s}}\otimes\rho_I^*\big(\cL\otimes \cO(-D_{i_1}-\ldots-D_{i_k})_{|D_1}\big)\big),$$
are either $0$ or they are supported on the divisors $\PP^{\frac{n}{2}-1}\times \PP^{\frac{n}{2}-1}$ as above 
(in particular, $n$ is even). In the latter case, writing $n=2s+2$, as both components of the above sheaves are acyclic, such objects are generated 
by $\cO(-a)\boxtimes\cO(-b)$ for $0<a,b\leq s$. We use here that if $\cA$ is an object in $D^b(\LM_{s+1})$ with $R\Ga(\cA)=0$
and $f:\LM_{s+1}\ra\PP^s$ is a Kapranov map, then $Rf_*\cA$ has the same property, therefore it is generated by 
$\cO(-a)$, for $0<a\leq s$. Using the above exact sequences, 
$$Rq_*\rho_I^*\big(\cL\otimes \cO(-D_{i_1}-\ldots-D_{i_k})_{|D_1}\big)$$
is either $0$ or, when $n$ is even, generated by 
$\cO(-a)\boxtimes\cO(-b)$ ($0<a,b\leq s$) on $\PP^s\times \PP^s$. 
Proposition \ref{pull-backs push} now follows.
\ep

\begin{lemma}[\emph{cf.} \protect{\cite[Lemma 4.6]{CT_partI}}]\label{easy acyclic}
Consider the divisor $D=-dH+\sum m_IE_I$ 
on $\LM_N$ written in some Kapranov model. The divisor $D$ is acyclic if 
$$1\leq d\leq n-3, \quad 0\leq m_I\leq n-3-|I|.$$
\end{lemma} 

The following lemma is well known:
\begin{lemma}\label{push exceptional}
Let $p: X\ra Y$ be a blow-up of a smooth subvariety $Z$ of codimension $r+1$ of a smooth projective variety $Y$. Let $E$ be the exceptional divisor. Then 
for all $1\leq i\leq r$ we have $Rp_*\cO_X(iE)=\cO_Y$. 
\end{lemma}
 
 \begin{lemma}\label{lift}
Let $Z=\LM_{N_1}\times\ldots\times\LM_{N_t}$ be a massive stratum in $\LM_N$ and let $D_1,\ldots, D_{t-1}$ be the boundary divisors whose 
intersection is $Z$. Let 
$$\cT=\cT_1\boxtimes\ldots\boxtimes\cT_t$$ be a sheaf supported on $Z$, with either $\cT_i=\cO$ or $\cT_i=G^\vee_{a_i}$, for some $1\leq a_i< |N_i|$, and not all $\cT_i=\cO$. Then there exists a line bundle $\cL$ on $\LM_N$ such that:
\bi
\item[(a) ] $\cL_{|Z}=\cT$; 
\item[(b) ] $\cL$ is acyclic;
\item[(c) ] For all $1\le i_1<\ldots<i_k\le t$, the restriction $\cL_{|D_{i_1}\cap\ldots\cap D_{i_k}}$ is acyclic. 
\ei
In addition, $\cL\otimes\cO(-D_{i_1}-\ldots-D_{i_k})$ is acyclic for all $1\le i_1<\ldots<i_k\le t$. 
\end{lemma} 

\bp 
The proof is by induction on $t\geq1$. The statement is trivially true when $t=1$, i.e., when $Z=\LM_n$ (as $\cL=\cT$ and there are no boundary divisors to be considered). In addition, if all but one of the $\cT_i$'s are trivial, say $\cT_i=G^\vee_{a_i}$, we are done by \cite[Lemma 4.3(3)]{CT_partI}, as we can take 
$$\cL=G^\vee_{a_i+|N_1|+\ldots+|N_{i-1}|}.$$ 

Assume now $t\geq2$ and at least two of the $\cT_i$'s are non-trivial. Consider 
$\pi_{N_1}: \LM_N\ra\LM_{N\setminus N_1}$ and let $Z'=\pi_{N_1}(Z)$. Then $Z'$ can be identified with $\LM_{N_2}\times\ldots\times\LM_{N_t}$ and the map $\pi_{N_1}: Z\ra Z'$ is the second projection. Let $\cT'=\cT_2\boxtimes\ldots\boxtimes\cT_t$. By induction, there is an acyclic line bundle $\cL'$ on $\LM_{N\setminus N_1}$ such that $\cL'_{|Z'}=\cT'$ and whose restriction to every stratum containing $Z'$ is also acyclic. If $\cT_1=\cO$, we let 
$\cL=\pi_{N_1}^*\cL'$ and clearly all of the properties are satisfied. If $\cT_1=G^\vee_a$, we define
$\cL=G^\vee_a\otimes\pi_{N_1}^*\cL'.$
Clearly, $\cL_{|Z}=\cT$. By the projection formula, $R{\pi_{N_1}}_*(\cL)=\cL'\otimes R{\pi_{N_1}}_*(G^\vee_a)$. As $R{\pi_i}_*(G^\vee_a)=0$ for all $i$, it follows that $R{\pi_{N_1}}_*(\cL)=0$,i.e., $\cL$ is acyclic. 

The same argument applies to show that the restriction of $\cL$ to a stratum $W$ containing $Z$ is acyclic. Consider such a stratum:
$$W=\LM_{M_1}\times\ldots\times\LM_{M_s},$$
and let $W'=\LM_{M_2}\times\ldots\times\LM_{M_s}$, considered as a stratum in in $\LM_{N\setminus M_1}$.
If $M_1=N_1$, the restriction $\cL_{|W}$ equals $G^\vee_a\boxtimes(\cL'_{|W'})$ and is clearly acyclic. 
If $M_1\neq N_1$, then $M_1=N_1+\ldots+N_i$, with $i\geq2$, and $\pi_{N_1}(W)$ is the stratum $\LM_{M_1\setminus N_1}\times W'$ in 
$\LM_{N\setminus N_1}$. The restriction of $\cL'$ to this stratum has the form $\cL'_1\boxtimes\cL'_2$. Then 
$\cL_{|W}=(G^\vee_a\otimes{\pi'}^*_{N_1}{\cL'_1})\boxtimes\cL'_2$, where $\pi'_{N_1}: \LM_{M_1}\ra\LM_{N_1}$ is the forgetful map. Again, by the projection formula, $\cL_{|W}$ is acyclic. 

We now prove the last assertion in the lemma. As $\cL$, $\cL_{D_i}$ are both acyclic, $\cL(-D_i)$ is acyclic (case $k=1$). The statement follows by induction on $k$ using the Koszul resolution for the intersection $\cap_{j\in I} D_j$,  $I=\{i_1,\ldots, i_k\}$:
$$\ldots \ra\oplus_{l<j, l,j\in I}\cO(-D_l-D_j)\ra  \oplus_{j\in I}\cO(-D_j)\ra\cO\ra\cO_{D_{i_1}\cap\ldots\cap D_{i_k}}\ra0.$$

\ep
 
Proposition \ref{pull-backs push} and Lemma \ref{dictionary} have the following:
\begin{cor}\label{rewrite} 
Assume $n=|N|$ is odd. Let $p:\LM_N\ra Z_N$ be the reduction map. 
For all $I\subseteq N$ with $0\leq |I|\leq n-2$ and all $1\leq a\leq n-|I|-1$, we have
$$Rp_*\big(\pi_I^*G_a^\vee\big)=\cO(-I^c)\otimes z^{2a-|I^c|},\quad I^c=N\setminus I.$$
Alternatively, this is the collection of $P\GG_m$-linearized line bundles 
$$\cO(-E)\otimes z^p,\quad 0\leq |p|\leq e-2,\quad 2\leq e\leq n\quad (e=|E|,\quad E\subseteq N).$$
Moreover, $Rp_*\cO=\cO$ and $Rp_*E=0$ for all other objects $E$ in the collection $\hat \GG$. 
\end{cor}

Proposition \ref{pull-backs push} and Lemma \ref{dictionary2} have the following:
\begin{cor}\label{rewrite2} 
Assume $|N|=2s+2$ is even. Let $p:\LM_N\ra Z_N$ be the reduction map. 
For all $E\subseteq N$, $e=|E|\geq2$ and all $1\leq a\leq e-1$,  
$$Rp_*\big(\pi_{N\setminus E}^*G_a^\vee\big)=\cO(-E)\big(\sum \left|a-|E\cap T^c|\right|E_T\big)\otimes z^{2a-e},$$
where $\left| a-|E\cap T^c|\right|$ denotes the absolute value of $(a-|E\cap T^c|)$. 
Moreover, $Rp_*\cO=\cO$. For all $G^\vee_a\otimes G^\vee_b$ supported on 
strata $\LM_{s+1}\times\LM_{s+1}$ we have
$$Rp_*\big(G^\vee_a\otimes G^\vee_b\big)=\cO(-a)\boxtimes\cO(-b)\quad (0<a, b\leq s),$$ 
 All other pushforwards are either $0$ or are generated by the above torsion sheaves. 
\end{cor}

When $n=4$, the map $p:\LM_N\ra Z_N$ is an isomorphism. In particular, the objects in $Rp_*\pi_I^*\hat\GG$ form a full exceptional collection.
However, it is straightforward to see that this is different than the collection in Theorem \ref{even}.



\begin{thebibliography}{BFK19}


\bibitem[BFK19]{BFK}
M. Ballard, D. Favero, and L. Katzarkov,
\emph{Variation of geometric invariant theory quotients and derived categories},
J. Reine Angew. Math. {\bf 746} (2019), 235--303.

\bibitem[BM13]{BergstromMinabe}
J. Bergstr\"om and S. Minabe,
\emph{On the cohomology of moduli spaces of (weighted) stable rational curves},
Math. Z. {\bf 275} (2013) no. 3-4, 1095--1108.

\bibitem[BM14]{BergstromMinabeLM}
J. Bergstr\"om and S. Minabe, 
\emph{On the cohomology of the Losev--Manin moduli space},
Manuscripta Mathematica {\bf 144} (2014), no. 1-2, 241--252.

\bibitem[CT12]{CT_rigid}
A.-M. Castravet and J. Tevelev,
\emph{Rigid curves on {$\M_{0,n}$} and arithmetic breaks}.
In: Compact moduli spaces and vector bundles, pp 19--67,
Contemp. Math., vol. 564, Amer. Math. Soc., Providence, RI, 2012.

\bibitem[CT13]{CT_Crelle}
A.-M. Castravet and J. Tevelev,
\emph{Hypertrees, projections, and moduli of stable rational curves},
J. Reine Angew. Math. {\bf 675} (2013), 121--180.

\bibitem[CT15]{CT_Duke}
A.-M. Castravet and J. Tevelev,
\emph{$\M_{0,n}$ is not a Mori dream space},
Duke Math. J. {\bf 164} (2015), no. 8, 1641--1667.

\bibitem[CT20a]{CT_partI}
A.-M. Castravet and J. Tevelev,
\emph{Derived category of moduli of pointed curves. I},
Algebr. Geom. {\bf 7} (2020), no. 6, 722--757.

\bibitem[CT20b]{CT_partII}
A.-M. Castravet and J. Tevelev,
\emph{Derived category of moduli of pointed curves. II},
preprint \arXiv{2002.02889} (2020).

\bibitem[Dol03]{Dolgachev}
I. Dolgachev,
\emph{Lectures on invariant theory},
London Mathematical Society Lecture Note Series, vol. 296,
Cambridge University Press, Cambridge, 2003.

\bibitem[HL15]{DHL}
D. Halpern-Leistner,
\emph{The derived category of a GIT quotient},
J. Amer. Math. Soc. {\bf 28} (2015), no. 3, 871--912.

\bibitem[Has03]{Ha}
B. Hassett,
\emph{Moduli spaces of weighted pointed stable curves},
Adv. Math. {\bf 173} (2003), no. 2, 316--352.

\bibitem[Huy06]{Huy}
D. Huybrechts,
\emph{Fourier-Mukai transforms in algebraic geometry},
Oxford Mathematical Monographs,
The Clarendon Press, Oxford University Press, Oxford, 2006.

\bibitem[KL09]{KeelTevelev}
S. Keel and J. Tevelev,
\emph{Equations for $\overline M_{0,n}$},
Internat. J. of Math. {\bf 20} (2009), no. 9, 1159--1184.

\bibitem[LM00]{LM}
A. Losev and Y. Manin,
\emph{New moduli spaces of pointed curves and pencils of flat connections},
Michigan Math. J. {\bf 48} (2000), 443--472.

\bibitem[MS13]{ManinSmirnov1}
Yu. I. Manin and M. N. Smirnov,
\emph{On the derived category of $\overline M_{0,n}$},
Izv. Ross. Akad. Nauk Ser. Mat. {\bf 77} (2013), no. 3, 93--108.

\bibitem[MS14]{ManinSmirnov2}
Yu. I. Manin and M. N. Smirnov,
\emph{Towards motivic quantum cohomology of $\overline{M}_{0,S}$},
Proc. Edinb. Math. Soc. (2) {\bf 57} (2014), no. 1, 201--230.

\bibitem[MP97]{MP}
A. S. Merkurjev and I. A. Panin,
\emph{$K$-theory of algebraic tori and toric varieties},
$K$-Theory {\bf 12} (1997), no. 2, 101--143.

\bibitem[Smi13]{Smirnov Thesis}
M. N. Smirnov,
\emph{Gromov-Witten correspondences, derived categories, and Frobenius manifolds},
Ph. D. thesis, University of Bonn, 2013. Available at
\href{https://bonndoc.ulb.uni-bonn.de/xmlui/handle/20.500.11811/5627}{https://bonndoc.ulb.uni-bonn.de/xmlui/handle/20.500.11811/5627}


\end{thebibliography}
\end{document}